\documentclass{article}
\usepackage{geometry}
\usepackage{amsmath}
\usepackage{latexsym}
\usepackage{amsthm}
\usepackage{amsfonts}
\usepackage{amssymb}
\usepackage{tikz}
\usepackage{tikz-cd}
\usepackage{gensymb}
\usepackage{cancel}
\usepackage{setspace}
\usepackage{mathtools}
\usepackage{framed}
\usepackage{tgbonum}
\usepackage{hyperref}
\usepackage{graphicx}
\usetikzlibrary{decorations.pathreplacing,angles,quotes}
\date{}
\author{Adam Afandi}
\title{Polynomiality of $\mathbb{Z}_2$ Hurwitz-Hodge Integrals}

\newtheorem{Definition}{Definition}
\newtheorem{Proposition}{Proposition}
\newtheorem{Lemma}{Lemma}
\newtheorem{Theorem}{Theorem}
\newtheorem{Corollary}{Corollary}
\newtheorem{Example}{Example}
\newtheorem{Remark}{Remark}
\newtheorem{Open Problem}{Open Problem}
\newtheorem{Conjecture}{Conjecture}
\newtheorem{Question}{Question}

\newcommand{\mbar}{\overline{\mathcal{M}}}
\newcommand{\bztwo}{\mathcal{B}\mathbb{Z}_2}
\newcommand{\aux}{\mbar_{0, (2g + 2)t}([\mathbb{P}^1/\mathbb{Z}_2], 1)}

\begin{document}
\maketitle

\begin{abstract}

Using Atiyah-Bott localization on the space of stable maps to the stack quotient $[\mathbb{P}^1/\mathbb{Z}_2]$, we find recursions that determine all Hodge integrals with descendent insertions at one marked point on the hyperelliptic locus $\overline{\mathcal{H}}_{g, 2g + 2} \subseteq \mbar_{g, 2g + 2}$. The initial conditions required for our recursions are gravitational descendents at one marked point, which are known to be $\frac{1}{2}$. We discover a new structure concerning these intersection numbers: for a fixed monomial of $\lambda$-classes, the resulting family of hyperelliptic Hodge integrals is polynomial in $g$. We formulate a conjecture concerning the log-concavity of the coefficients of these polynomials. Lastly, we turn our recursions into a non-linear system of partial differential equations for the generating functions of hyperelliptic Hodge integrals.  

\end{abstract}

\tableofcontents

\pagebreak


\section{Introduction}

Let $\mbar_{g, n}(\mathcal{B}G)$ be the moduli space of $n$-pointed genus $g$ stable maps into the classifying space of a finite group $G$. Interest in understanding the intersection theory of this space are twofold. The first impetus comes from the fact that these spaces are compactifications of Hurwitz spaces via admissible covers. Viewing these spaces from the perspective of stable maps into stacky points provides an alternative approach in understanding their intersection theory. The remaining reason for their study is due to the fact that when one computes \emph{twisted} Gromov-Witten invariants of an orbifold $\mathcal{X}$ using torus localization, one needs to compute \emph{Hurwitz-Hodge integrals} i.e. intersection numbers on $\mbar_{g, n}(\mathcal{B}G)$ that involve Hodge classes and $\psi$-classes. 


\subsection{Summary of Results}

Let $\mbar_{0, (2g + 2)t}(\bztwo)$ be the moduli space of stable maps from genus 0 orbifold curves to the classifying space $\bztwo := [pt./\mathbb{Z}_2]$. The `t' decoration on the marked points indicate that all of the points on the orbifold curve have non-trivial $\mathbb{Z}_2$-isotropy (they are `twisted'). This space is naturally isomorphic to $\overline{\mathcal{H}}_{g, 2g + 2} \subseteq \mbar_{g, 2g + 2}$, the moduli space of stable hyperelliptic curves with $2g + 2$ marked Weierstrass points. We define $\lambda$-classes on $\mbar_{0, (2g + 2)t}(\bztwo)$ as the pullback of the standard $\lambda$-classes $\lambda_i := c_i(\mathbb{E}_g)$ on $\mbar_{g, 2g + 2}$. There is also the map $\overline{\mathcal{H}}_{g, 2g + 2} \rightarrow \mbar_{0, 2g + 2}$ that takes a hyperelliptic covering and remembers its rational base curve, along with the branch locus. We define $\psi$-classes on $\mbar_{0, (2g + 2)t}(\bztwo)$ as the pullback of the standard $\psi$-classes $\psi_i := c_1(\mathbb{L}_i)$ on $\mbar_{0, 2g + 2}$. Similarly, we let $\mbar_{0, (2g + 2)t, 1u}(\bztwo)$ be the moduli of stable maps as before, but with an extra marked point on the orbifold curve with trivial isotropy (an `untwisted' point).

Our results are easier to state using the following multi-index notation: for a tuple of non-negative integers $\vec{i} := (i_1, \ldots, i_n) \in \mathbb{Z}_{\geq 0}^n$, define $|\vec{i}| = i_1 + \ldots + i_n$. At the heart of this paper are the following three Theorems:

\begin{Theorem}\label{Recursion}
Let $1 \leq j \leq 2g + 2$, and let $\vec{i} = (i_1, \ldots, i_n) \in \mathbb{Z}_{\geq 0}^n$. Then there exist recursions that determine all Hurwitz-Hodge integrals of the form

\begin{align}
\int_{\mbar_{0, (2g + 2)t}(\bztwo)}(\psi_j)^{2g - 1 - |\vec{i}|} & \lambda_{i_1}\ldots\lambda_{i_n} \label{twisted} \\
\int_{\mbar_{0, (2g + 2)t, 1u}(\bztwo)}(\psi_{2g + 3})^{2g - |\vec{i}|} & \lambda_{i_1}\ldots\lambda_{i_n}. \label{untwisted}
\end{align}

\noindent The only initial conditions that the recursions require are the following descendent integrals:

\begin{equation*}
\int_{\mbar_{0, (2g + 2)t}(\bztwo)}\psi_j^{2g - 1} = \int_{\mbar_{0, (2g + 2)t, 1u}(\bztwo)}\psi_{2g + 3}^{2g} = \frac{1}{2}.
\end{equation*}
\end{Theorem}

\begin{Theorem}\label{Integrality}
The Hurwitz-Hodge integrals in Equation \ref{twisted} and Equation \ref{untwisted} are integers after a suitable normalization. Specifically, for all $\vec{i} = (i_1, \ldots, i_n)$ and $g$,

\begin{align*}
& 2^{|\vec{i}| + 1} \int_{\mbar_{0, (2g + 2)t}(\bztwo)}(\psi_j)^{2g - 1 - |\vec{i}|} \lambda_{i_1}\ldots\lambda_{i_n} \in \mathbb{Z} \\
& 2^{|\vec{i}| + 1} \int_{\mbar_{0, (2g + 2)t, 1u}(\bztwo)}(\psi_{2g + 3})^{2g - |\vec{i}|} \lambda_{i_1}\ldots\lambda_{i_n} \in \mathbb{Z}.
\end{align*}

\end{Theorem}

\begin{Theorem}\label{Polynomiality}
For a fixed $\vec{i} = (i_1, \ldots, i_n)$, the Hurwitz-Hodge integrals in Equation \ref{twisted} and Equation \ref{untwisted} are polynomial in $g$. The degree of these polynomials is bounded above by $|\vec{i}|^2 + 1$. Furthermore, there is an algorithm to compute these polynomials recursively (see Corollary \ref{binomialbasis}). 
\end{Theorem}

\noindent Finally, we translate the recursions obtained in Theorem \ref{Recursion} into a system of non-linear partial differential equations for the generating functions of $\mathbb{Z}_2$ Hurwitz-Hodge integrals:

\begin{Theorem}\label{PDE}
Let $s_1, s_2, \ldots$ and $t$ be formal variables. For $\vec{i} = (i_1, \ldots i_n)$, define $s^{\vec{i}} := s_1^{i_1}\ldots s_n^{i_n}$. Let $F$ and $G$ be the following generating functions for $\mathbb{Z}_2$ Hurwitz-Hodge integrals, 

\begin{align*}
& F(s_1, s_2, \ldots ; t) = F(\vec{s}, t) := \sum_{\substack{g \geq 0 \\ \vec{i} \geq \vec{0}}} \left(\int_{\mbar_{0, (2g + 2)t}(\bztwo)}(\psi_j)^{2g - 1 - |\vec{i}|} \lambda_{i_1}\ldots\lambda_{i_n}\right)s^{\vec{i}}\frac{t^{2g + 2}}{(2g + 2)!} \\
& G(s_1, s_2, \ldots ; t) = G(\vec{s}, t) := \sum_{\substack{g \geq 0 \\ \vec{i} \geq \vec{0}}}\left(\int_{\mbar_{0, (2g + 2)t, 1u}(\bztwo)}(\psi_{2g + 3})^{2g - |\vec{i}|} \lambda_{i_1}\ldots\lambda_{i_n}\right)s^{\vec{i}}\frac{t^{2g + 2}}{(2g + 2)!}
\end{align*}

\noindent Then $F$ and $G$ satisfy the following third order non-linear system of partial differential equations:

\begin{align}
& 2\partial_t^3F(\vec{s}, t)\partial_t^2F(-\vec{s}, t) = 2\partial_t^2G(\vec{s}, t)\partial_tG(-\vec{s}, t)  \label{thepdes1} \\
& 2\partial_t^2G(\vec{s}, t)G(-\vec{s}, t) = 2\partial_t^3F(\vec{s}, t)\partial_tF(-\vec{s}, t) - \partial_t^2G(\vec{s}, t) \label{thepdes2}
\end{align}
\end{Theorem}

\subsection{Background and Context}

\noindent The earliest result concerning these integrals goes back to Faber and Pandharipande in \cite{FP2000}, in which they discovered that

\begin{equation*}
\int_{\mbar_{0, (2g + 2)t}(\bztwo)}\lambda_{g - 1}\lambda_g = \frac{2^{2g} - 1}{2g}\cdot |B_{2g}|
\end{equation*}

\noindent where $B_{2g}$ is the $2g^{th}$ Bernoulli number. In \cite{Cavalieri2008}, Cavalieri generalized Faber and Pandharipande's result by considering integrals of the form

\begin{equation*}
\int_{\mbar_{0, (2g + 2)t}(\bztwo)}\psi_j^{i - 1}\lambda_{g - i}\lambda_g
\end{equation*} 

\noindent Cavalieri packaged these integrals into generating functions and found recursions for them. In \cite{Wise2016}, Wise showed that

\begin{equation*}
\int_{\mbar_{0, (2g + 2)t}(\bztwo)}\frac{c(\mathbb{E}_g^\vee)^2}{1 - \frac{\psi_1}{2}} = \left(-\frac{1}{4}\right)^g
\end{equation*}

\noindent In \cite{AbelianHurwitzHodge2011}, Johnson-Pandharipande-Tseng found an algorithm to express $G$ Hurwitz-Hodge integrals with one $\lambda$ class in terms of Hurwitz numbers. Most recently, in \cite{Afandi2019}, the author found closed form expressions for $\mathbb{Z}_2$ Hurwitz-Hodge integrals with one $\lambda$ insertion:

\begin{align*}
& \int_{\mbar_{0, (2g + 2)t}(\bztwo)}\psi_1^{2g - 1 - i}\lambda_i = \left(\frac{1}{2}\right)^{i + 1}e_i(g) \\
& \int_{\mbar_{0, (2g + 2)t, 1u}(\bztwo)}\psi_{2g + 3}^{2g - i}\lambda_i = \left(\frac{1}{2}\right)^{i + 1}\hat{e}_i(g)
\end{align*}

\noindent where $e_i(g) := e_i(1, 3, \ldots, 2g - 1)$ and $\hat{e}_i(g) := e_i(2, 4, \ldots, 2g)$, and $e_i$ is the $i^{th}$ elementary symmetric function. 

Although the computation of Hurwitz-Hodge integrals are of substantial interest and significance, there are very few results that indicate efficient ways to evaluate them. They have a reputation of being a class of   intersection numbers that are notoriously difficult to compute. Two common approaches used to tackle them are Atiyah-Bott localization, and the orbifold Grothendieck-Riemann-Roch (GRR) theorem. In both of these approaches, the main obstacle is usually an inability to tame the combinatorial complexity of the calculations. 

In addition to localization and GRR, one can use the \emph{mirror theorem} for the toric orbifolds $[\mathbb{C}^n/G]$ to compute Hurwitz-Hodge integrals. Established in \cite{CCIT2009}, the mirror theorem finds a connection between two seemingly unrelated functions, the twisted $I$-function, and the twisted $J$-function. The twisted $J$-function is a cohomology-valued generating function for Hurwitz-Hodge integrals, whereas the twisted $I$-function is a hypergeometric series. The mirror theorem says that these two functions coincide after a change of variables, along with a complicated procedure involving Birkhoff factorization. In order to extract closed form expressions for Hurwitz-Hodge integrals from the twisted $J$-function, one must find an expression for the inverse of the \emph{mirror map}. This allows one to invert the formal generating function parameters between the twisted $I$-function and the twisted $J$-function. In theory, this method of computing Hurwitz-Hodge integrals is incredibly powerful due to its breadth and scope of application. However, in practice, using the mirror theorem to compute Hurwitz-Hodge has proven to be difficult. 

In the language of \cite{CCIT2009}, this paper computes the genus 0 twisted orbifold Gromov-Witten theory of the stack quotient $[\mathbb{C}^n/\mathbb{Z}_2]$ for all $n$, without appealing to the mirror theorem. The significance of our work stems from the fact that, as far as the author is aware, there are no closed form expressions for the mirror map and its inversion for $[\mathbb{C}^n/\mathbb{Z}_2]$ when $n \geq 3$. We circumvent the mirror theorem altogether by directly tackling the recursions obtained when using torus localization. Our approach is computationally involved, however, we are successful in our combinatorial analysis. We discover a rich combinatorial structure (see Theorem \ref{Integrality} and Theorem \ref{Polynomiality}) for Hurwitz-Hodge integrals that does not obviously follow from the mirror theorem. The exciting part of this development is that, we now know what kind of structure to look for as we generalize beyond the hyperelliptic case. The most optimistic outlook is that the combinatorial formulas we see in the case when $g = 0$ and $G = \mathbb{Z}_2$ is simply a shadow of a much more general phenomenon as $g$ increases, and as $G$ varies.


\subsection{Plan of the Paper}

\noindent In Section \ref{preliminaries}, we establish all of the relevant background and definitions. Those already familiar with orbifold Gromov-Witten theory/moduli of curves can safely skip this section. Section \ref{localization} is the computational heart of the paper, and sets up the localization computations that allow us to prove Theorems \ref{Recursion}, \ref{Integrality}, and \ref{Polynomiality}. Section \ref{proofs} is the main part of the paper, in which we provide the proofs of all the main results. We end the paper by working through an example explicitly (Section \ref{example}), provide some conjectures for future work (Section \ref{conjectures}), and we give evidence for these conjectures in the Appendix.  

\section*{Acknowledgements}
The author is grateful for his thesis advisors, Mark Shoemaker and Renzo Cavalieri, for their advice, and their meticulous reading of early drafts. Conversations with Maria Gillespie helped clarify many of the  combinatorial constructions in this paper. \\

\noindent This work was partially supported by NSF grant DMS-1708104.


\section{Moduli of Stable Maps Into $\mathcal{B}\mathbb{Z}_2$}\label{Moduli of Stable Maps}\label{preliminaries}

This section establishes definitions and notation used throughout the paper. We provide only the minimum amount of mathematical machinery in order to understand the rest of the paper. We recommend  \cite{Afandi2019}, Section 2, for a more thorough discussion. 

\begin{Definition}\label{moduli}
Let $\mbar_{0, (2g + 2)t}(\bztwo)$ denote the moduli space of $(2g + 2)$-pointed genus $0$ stable maps into $\bztwo$, whose marked points are mapped into the twisted sector of the inertia stack $\mathcal{I}\mathcal{B}\mathbb{Z}_2$. Similarly, let $\mbar_{0, (2g + 2)t, 1u}(\bztwo)$ be the moduli space of $2g + 3$-pointed genus 0 stable maps into $\bztwo$, where the first $2g + 2$ marked points map into the twisted sector of $\mathcal{I}\mathcal{B}\mathbb{Z}_2$, and the last marked point maps into the untwisted sector. Lastly, we let $\mbar_{0, (2g + 2)t}([\mathbb{P}^1/\mathbb{Z}_2], 1)$ be the moduli space of $2g + 2$-pointed genus 0 stable maps of degree 1 into the stack quotient $[\mathbb{P}^1/\mathbb{Z}_2]$, where $\mathbb{Z}_2$ has the trivial action on $\mathbb{P}^1$.
\end{Definition}

There is a natural isomorphism between $\mbar_{0, (2g + 2)t}(\bztwo)$ and $\overline{\mathcal{H}}_{g, 2g + 2} \subseteq \mbar_{g, 2g + 2}$, the moduli space of $2g + 2$-pointed genus $g$ stable hyperelliptic curves, whose marked points are Weierstrass points. Similarly, $\mbar_{0, (2g + 2)t, 1u}(\bztwo)$ is isomorphic to $\overline{\mathcal{H}}_{g, 2g + 2, 2}\subseteq \mbar_{g, 2g + 3}$, the same moduli of hyperelliptic curves as before, except we also mark two points exchanged under the hyperelliptic involution. 

The spaces of stable hyperelliptic curves come equipped with the rank $g$ Hodge bundle $\mathbb{E}_g$, which we can pull back to $\mbar_{0, (2g + 2)t}(\bztwo)$ and $\mbar_{0, (2g + 2)t, 1u}(\bztwo)$. We also have maps $\mbar_{0, (2g + 2)t}(\bztwo) \rightarrow \mbar_{0, 2g + 2}$ and $\mbar_{0, (2g + 2)t, 1u}(\bztwo) \rightarrow \mbar_{0, 2g + 3}$ which forget the orbifold structure at the marked points. With this, we make the following definition:

\begin{Definition}
Whenever we refer to the classes $\lambda_i := c_i(\mathbb{E}_g)$ and $\psi_i := c_1(\mathbb{L}_i)$ in the Chow ring of $\mbar_{0, (2g + 2)t}(\bztwo)$ or $\mbar_{0, (2g + 2)t, 1u}(\bztwo)$, we mean the class $\lambda_i$ pulled back from the hyperelliptic locus, and the $\psi$-class pulled back from the moduli of pointed rational curves.
\end{Definition}   

There is an analagous description for $\lambda$-classes and $\psi$-classes on $\mbar_{0, (2g + 2)t}([\mathbb{P}^1/\mathbb{Z}_2], 1)$. Recall from Definition \ref{moduli} that we impose the trivial $\mathbb{Z}_2$-action on $\mathbb{P}^1$. This means that $[\mathbb{P}^1/\mathbb{Z}_2] = \mathbb{P}^1 \times \bztwo$. Given a moduli point $[(C, p_1, \ldots, p_{2g + 2}) \rightarrow [\mathbb{P}^1/\mathbb{Z}_2]] \in \mbar_{0, (2g + 2)t}([\mathbb{P}^1/\mathbb{Z}_2], 1)$, one can post-compose the map $C \rightarrow [\mathbb{P}^1/\mathbb{Z}_2]$ with the canonical projection map $[\mathbb{P}^1/\mathbb{Z}_2] = \mathbb{P}^1 \times \mathcal{B}\mathbb{Z}_2 \rightarrow \mathcal{B}\mathbb{Z}_2$ which is precisely the datum of a point in $\mbar_{0, (2g + 2)t}(\bztwo)$. In other words, there is a map $\mbar_{0, (2g + 2)t}([\mathbb{P}^1/\mathbb{Z}_2], 1) \rightarrow \mbar_{0, (2g + 2)t}(\bztwo)$. The $\lambda$-classes on $\mbar_{0, (2g + 2)t}([\mathbb{P}^1/\mathbb{Z}_2], 1)$ are pulled back along this map. The $\psi$-classes on $\mbar_{0, (2g + 2)t}([\mathbb{P}^1/\mathbb{Z}_2], 1)$ are defined by taking the first Chern class of the line bundle whose fibers are the cotangent spaces at the $i^{th}$ marked points of the source curve $C$.

We recall the \emph{orbifold Riemann-Roch Theorem} (see \cite{AGV2008}), applied to orbifold curves and line bundles. Let $L$ be a line bundle over an orbifold curve $\mathcal{X}$ of genus $g$, and suppose $\mathcal{X}$ has only finitely many points $p_1, \ldots, p_n \in \mathcal{X}$ that have non-trivial isotropy. If $G_{p_i}$ is the local group at $p_i$, then $G_{p_i}$ acts on $L_{p_i}$. This action takes the form $z \mapsto e^{\frac{2\pi i}{r}}z$, where $r \in \{1, 2, \ldots, |G_{p_i}|\}$. We define $\text{age}_{p_i}(L) := \frac{1}{r}$, and so the Euler characteristic of $L$ is

\begin{equation*}
\chi(L) = (1 - g) + \text{deg}(L) - \sum_{i = 1}^n\text{age}_{p_i}(L)
\end{equation*}

\noindent Consider the line bundle $\mathcal{O}_{\mathbb{P}^1}(-1)$ over $[\mathbb{P}^1/\mathbb{Z}_2]$. By this we mean, take $\mathcal{O}(-1)$ on $\mathbb{P}^1$, and pullback along the map $[\mathbb{P}^1/\mathbb{Z}_2] \rightarrow \mathbb{P}^1$ that forgets the orbifold structure, and let $\mathbb{Z}_2$ act non-trivially on the fibers. Let $[f: (C, p_1, \ldots, p_{2g + 2})\rightarrow [\mathbb{P}^1/\mathbb{Z}_2]] \in \mbar_{0, (2g + 2)t}([\mathbb{P}^1/\mathbb{Z}_2], 1)$. Then by orbifold Riemann-Roch, we have

\begin{equation*}
\chi(f^*\mathcal{O}_{\mathbb{P}^1}(-1)) = -\sum_{i = 1}^{2g + 2}\frac{1}{2} = - g - 1
\end{equation*}

\noindent Since $h^0(f^*\mathcal{O}_{\mathbb{P}^1}(-1)) = 0$, we have $h^1(f^*\mathcal{O}_{\mathbb{P}^1}(-1)) = g + 1$. This observation justifies the following:

\begin{Proposition}
Let $\pi$ be the universal family over $\mbar_{0, (2g + 2)t}([\mathbb{P}^1/\mathbb{Z}_2], 1)$,

\begin{center}
\begin{tikzcd}
\mathcal{C} \arrow[d, "\pi"] \arrow[r, "f"] & {[\mathbb{P}^1/\mathbb{Z}_2]} \\
\mbar_{0, (2g + 2)t}([\mathbb{P}^1/\mathbb{Z}_2], 1)
\end{tikzcd}
\end{center}

\noindent Then $R^1\pi_*f^*\mathcal{O}_{\mathbb{P}^1}(-1)$ is a vector bundle over $\mbar_{0, (2g + 2)t}([\mathbb{P}^1/\mathbb{Z}_2], 1)$. The fiber over $[f : (C, p_1, \ldots, p_{2g + 2} \rightarrow [\mathbb{P}^1/\mathbb{Z}_2]]$ is $H^1(C, f^*\mathcal{O}_{\mathbb{P}^1}(-1))$.

\end{Proposition}

Lastly, we recall a result due to Jarvis and Kimura \cite{JK2001}. Later in the paper, we find recursions for $\mathbb{Z}_2$ Hurwitz-Hodge integrals, and the initial conditions for these recursions come from the following result:

\begin{Proposition}\cite{JK2001}
Intersection numbers only involving descendent insertions ($\psi$-classes) on $\mbar_{0, (2g + 2)t}(\bztwo)$ and $\mbar_{0, (2g + 2)t, 1u}(\bztwo)$ are completely determined. In particular

\begin{equation*}
\int_{\mbar_{0, (2g + 2)t}(\bztwo)}\psi_i^{2g - 1} = \int_{\mbar_{0, (2g + 2)t, 1u}(\bztwo)}\psi_{2g + 3}^{2g} = \frac{1}{2}
\end{equation*}
\end{Proposition} 


\section{Localization}\label{localization}


\subsection{Setup}

In this section, we describe how to use Atiyah-Bott localization to find recursions for $\mathbb{Z}_2$ Hurwitz-Hodge integrals.

In short, Atiyah-Bott localization is a standard computational technique to compute intersection numbers on a space that has a torus action. For the original exposition on localization, see \cite{AtiyahBott1994}, and for the case of integration over the moduli of stable maps, see \cite{GraberPandharipande1999}, \cite{Melissa2011}. 

We briefly recall how localization is adapted in the particular case of $\aux$. For a more detailed discussion, see \cite{Afandi2019}, Section 3. If $\gamma \in \mathbb{C}^*$, we let $\mathbb{C}^*$ act on $\mathbb{P}^1$ via $\gamma \cdot [x_0 : x_1] = [x_0 : \gamma x_1]$. This action commutes with the trivial action of $\mathbb{Z}_2$ on $\mathbb{P}^1$, so this $\mathbb{C}^*$-action lifts to the stack quotient $[\mathbb{P}^1/\mathbb{Z}_2]$. If we post compose this action with each stable map in $\aux$, we get a $\mathbb{C}^*$-action on $\aux$, and therefore, we can localize. The localization formula is a sum over the components of the $\mathbb{C}^*$-fixed loci of $\aux$. The $\mathbb{C}^*$-fixed points in $\aux$ are precisely the maps that send all of the points with non-trivial $\mathbb{Z}_2$-isotropy to the $\mathbb{C}^*$-fixed points of $[\mathbb{P}^1/\mathbb{Z}_2]$, namely, $0 := [1 : 0]$ and $\infty := [0 : 1]$. Since these maps are of degree 1, these $\mathbb{C}^*$-fixed points of $\aux$  necessarily contract rational components to $0$ and $\infty$, where the components contain all of the marked points. These two components are nodal to a $\mathbb{P}^1$ that maps isomorphically onto $[\mathbb{P}^1/\mathbb{Z}_2]$. 

With the above description, the components of the $\mathbb{C}^*$-fixed loci of $\aux$ are in bijection with \emph{localization graphs} of $\aux$ (see \cite{Afandi2019} Section 3). Localization graphs for $\aux$ are trees with two vertices, $v_0$ and $v_{\infty}$. The vertex $v_0$ represents a contracted component over $0$ and $v_{\infty}$ a contracted component over $\infty$. The vertices are connected by an edge, which represents the $\mathbb{P}^1$ mapping isomorphically onto $[\mathbb{P}^1/\mathbb{Z}_2]$. Lastly, each vertex is incident to a marked half-edge. Each half-edge corresponds to the marked points. If $\Gamma$ is a localization graph of $\aux$, then we associate to $\Gamma$ a product of moduli spaces $\mbar_{\Gamma} := \mbar_{v_0} \times \mbar_{v_\infty} \subseteq \aux$, where $\mbar_{v_0}$ and $\mbar_{v_{\infty}}$ are spaces of stable maps to $\bztwo$. 

In the present context, Atiyah-Bott localization is summed up in the following way:

\begin{Proposition}
Let $\alpha \in A^*(\mbar_{0, (2g + 2)t}([\mathbb{P}^1/\mathbb{Z}_2], 1)$. Then

\begin{equation*}
\int_{\mbar_{0, (2g + 2)t}([\mathbb{P}^1/\mathbb{Z}_2], 1)} \alpha = \sum_{\Gamma}\int_{\mbar_{\Gamma}}\frac{\alpha\vert_{\mbar_{\Gamma}}}{e(N_\Gamma)}
\end{equation*}

\noindent where the sum is over all localization graphs $\Gamma$ of $\mbar_{0, (2g + 2)t}([\mathbb{P}^1/\mathbb{Z}_2], 1)$, $e(N_{\Gamma})$ is the Euler class to the normal bundle to $\mbar_{\Gamma}$, and $\mbar_{\Gamma}$ is the corresponding moduli space for the localization graph $\Gamma$.

\end{Proposition} 

A key step in localization is the computation of $\frac{1}{e(N_{\Gamma})}$. We refer the reader to \cite{GraberPandharipande1999}, \cite{MS03} for more details about how to derive a general formula, but in our case, there is a simple description. Throughout, we let $t$ be the equivariant parameter of the $\mathbb{C}^*$-equivariant Chow ring of $\aux$ (see \cite{Afandi2019}, Section 3). For every $\Gamma$, $\frac{1}{e(N_{\Gamma})}$ has a factor of $\frac{-1}{t^2}$. If $\mbar_{\Gamma}$ contains maps that contract a component over $0$, then $\frac{1}{e(N_{\Gamma})}$ has a factor of $\frac{1}{t - \psi_{v_0}}$ on the space $\mbar_{v_0}$, where $\psi_{v_0}$ is the $\psi$-class at the node. Similarly, if $\mbar_{\Gamma}$ contains maps that contract a component over $\infty$, $\frac{1}{e(N_{\Gamma})}$ contains a factor of $\frac{1}{-t - \psi_{v_{\infty}}}$. Finally, due to the presence of gerbe structure, we also require a \emph{gluing factor} in the calculation of $\frac{1}{e(N_\Gamma)}$ (see \cite{CC09}). In our case, the gluing factor is $\left(\frac{1}{2}\right)^{T_{\Gamma}}\left(\frac{1}{2}\right)\cdot (2)^{n_{\Gamma}}$, where $T_{\Gamma}$ is the number of contracted components with exactly two marked points, and $n_{\Gamma}$ is the number of nodes.

Our localization computations make repeated use of the formulas in Equations \ref{restrictlambda}, \ref{restrictzero}, and \ref{restrictinfinity}. These formulas indicate how $\lambda_{\vec{i}} := \lambda_{i_1}\ldots\lambda_{i_n}, \text{ev}_i^*(\infty)$, and $\text{ev}_ i^*(0)$ restrict to $\mbar_{\Gamma}$:

\begin{align}
& \lambda_{\vec{i}}\vert_{\mbar_{\Gamma}} = \sum_{\vec{i}_1 + \vec{i}_2 = \vec{i}}\lambda_{\vec{i}_1}\vert_{\mbar_{v_0}} \times \lambda_{\vec{i}_2}\vert_{\mbar_{v_{\infty}}} \label{restrictlambda}\\
& \text{ev}_i^*(0)\vert_{\mbar_{\Gamma}} = t \label{restrictzero} \\
& \text{ev}_i^*(\infty)\vert_{\mbar_{\Gamma}} = -t \label{restrictinfinity}
\end{align}

Lastly, we describe how the vector bundle bundle $R^1\pi_*f^*\mathcal{O}_{\mathbb{P}^1}(-1)$ restricts to the $\mathbb{C}^*$-fixed loci. In order to localize, we make a choice of \emph{torus weights} on $\mathcal{O}_{\mathbb{P}^1}(-1)$ (see \cite{MS03} for more details). This amounts to determining how $\mathbb{C}^*$ acts on $\mathcal{O}_{\mathbb{P}^1}(-1)\vert_0$ and $\mathcal{O}_{\mathbb{P}^1}(-1)\vert_{\infty}$. Throughout, we choose the torus weights to be $a$ on $\mathcal{O}_{\mathbb{P}^1}(-1)\vert_0$ and $a + 1$ on $\mathcal{O}_{\mathbb{P}^1}(-1)\vert_{\infty}$. In other words, for $\gamma \in \mathbb{C}^*$, the torus action on the fiber over $0$ is $\gamma \cdot z = \gamma^az$, and over $\infty$ it is $\gamma \cdot z = \gamma^{a + 1}z$.

Let $\Gamma$ be a localization graph of $\mbar_{0, (2g + 2)t}([\mathbb{P}^1/\mathbb{Z}_2], 1)$, and let $[f : (C, p_1, \ldots, p_{2g + 2}) \rightarrow ~ [\mathbb{P}^1/\mathbb{Z}_2]]$ be a representative point in the loci corresponding to $\Gamma$. We denote $C_{v_0}$ and $C_{v_\infty}$ as the contracted components over $0$ and $\infty$, respectively. We assume $C_{v_0}$ and $C_{v_\infty}$ are of positive dimension. Consider the \emph{normalization exact sequence} for $C$,

\begin{center}
\begin{tikzcd}
0 \arrow[r] & \mathcal{O}_C \arrow[r] & \mathcal{O}_{v_0} \oplus \mathcal{O}_{v_\infty} \oplus \mathcal{O}_{e} \arrow[r] & \mathcal{O}_{n_0} \oplus \mathcal{O}_{n_\infty} \arrow[r] & 0
\end{tikzcd}
\end{center}

\noindent where

\begin{itemize}
\item{$\mathcal{O}_{v_0}$ is the structure sheaf on $C_{0}$, the contracted component over $0$} 
\item{$\mathcal{O}_{v_\infty}$ is the structure sheaf on $C_{\infty}$, the contracted component over $\infty$}
\item{$\mathcal{O}_{e}$ is the structure sheaf on $C_e$, the rational curve nodal to the contracted components}
\item{$\mathcal{O}_{n_0}$ is the skyscraper sheaf supported on $n_0$, the node above $0$}
\item{$\mathcal{O}_{n_\infty}$ is the skyscraper sheaf supported on $n_\infty$, the node above $\infty$}
\end{itemize}

\noindent Tensoring this sequence with $f^*\mathcal{O}_{\mathbb{P}^1}(-1)$, and applying cohomology, we get the long exact sequence

\begin{align*}
0 \longrightarrow & H^0(C, f^*\mathcal{O}_{\mathbb{P}^1}(-1)) \longrightarrow H^0(C_0, \mathbb{C} \times C_0) \oplus H^0(C_\infty, \mathbb{C} \times C_\infty) \oplus H^0(C_e, \mathcal{O}_{\mathbb{P}^1}(-1)) \\
\vspace{1cm} \\
\longrightarrow & H^0(n_0, \mathbb{C} \times n_0) \oplus H^0(n_\infty, \mathbb{C} \times n_\infty) \longrightarrow H^1(C, f^*\mathcal{O}_{\mathbb{P}^1}(-1)) \\
\vspace{1cm} \\
\longrightarrow & H^1(C_0, \mathbb{C} \times C_0) \oplus H^1(C_\infty, \mathbb{C} \times C_\infty) \oplus H^1(C_e, f^*\mathcal{O}_{\mathbb{P}^1}(-1)) \longrightarrow 0
\end{align*}

\noindent It is clear that $H^0(C, f^*\mathcal{O}_{\mathbb{P}^1}(-1)) = H^0(C_e, \mathcal{O}_{\mathbb{P}^1}(-1)) = 0$. However, we also claim that $H^0(C_0, \mathbb{C} \times C_0) = H^0(C_\infty, \mathbb{C} \times C_\infty) = 0$. To see this let $s : C_0 \rightarrow \mathbb{C} \times C_0$. be a constant section. The section $s$ must be $\mathbb{Z}_2$-equivariant, and in particular, if $p_i \in C_0$ is a point with non-trivial isotropy, then $s(p_i) = -s(p_i)$, and therefore, $s$ is the zero section. The same argument works if we choose a section $s: C_\infty \rightarrow \mathbb{C} \times C_\infty$. In summary, the above sequence reduces to the short exact sequence

\begin{align*}
0 \longrightarrow & H^0(n_0, \mathbb{C} \times n_0) \oplus H^0(n_\infty, \mathbb{C} \times n_\infty) \longrightarrow H^1(C, f^*\mathcal{O}_{\mathbb{P}^1}(-1)) \\
\vspace{1cm} \\
\longrightarrow & H^1(C_0, \mathbb{C} \times C_0) \oplus H^1(C_\infty, \mathbb{C} \times C_\infty) \oplus H^1(C_e, f^*\mathcal{O}_{\mathbb{P}^1}(-1)) \longrightarrow 0
\end{align*}

Lets first consider the case when the corresponding moduli space for $\Gamma$ is of the form $\mbar_{0, k_1t, 1u}(\bztwo) \times \mbar_{0, k_2t, 1u}(\bztwo)$. Define $g_1 := \frac{k_1 - 2}{2}$ and $g_2 := \frac{k_2 - 2}{2}$. If $s: n_0 \rightarrow \mathbb{C} \times n_0$ is an element of $H^0(n_0, \mathbb{C}\times n_0$, then since $n_0$ has trivial isotropy, imposing $\mathbb{Z}_2$ equivariance on $s$ does nothing. Therefore, $s$ can be any constant section. A similar argument holds for $n_\infty$.  Therefore, normalization exact sequence becomes

\begin{align}\label{exactuntwisted}
0 \longrightarrow L_a \oplus L_{a + 1} \rightarrow H^1(C, f^*\mathcal{O}_{\mathbb{P}^1}(-1)) \longrightarrow \mathbb{E}^\vee_{g_1} \oplus \mathbb{E}^\vee_{g_2} \longrightarrow 0
\end{align}

\noindent where $L_a$ is defined to be the trivial equivariant line bundle with torus weight $a$ (see Definition \ref{equivariantlinebundle}). Now consider the case when $\Gamma$ corresponds to a moduli space of the form $\mbar_{0, k_1t}(\bztwo) \times \mbar_{0, k_2t}(\bztwo)$. We have $H^0(n_0, \mathbb{C} \times n_0) = H^0(n_\infty, \mathbb{C} \times n_\infty) = 0$. Indeed, if $s \in H^0(n_0, \mathbb{C} \times n_0)$, then since $n_0$ has non-trivial isotropy, $\mathbb{Z}_2$ equivariance forces $s$ to be the zero section. The same holds for $n_\infty$. In which case, we get the very simple two term sequence

\begin{align}\label{exacttwisted}
0 \rightarrow H^1(C, f^*\mathcal{O}_{\mathbb{P}^1}(-1)) \rightarrow \mathbb{E}^\vee_{g_1} \oplus \mathbb{E}^\vee_{g_2} \oplus L_{a + \frac{1}{2}} \longrightarrow 0
\end{align} 

\noindent The term $L_{a + \frac{1}{2}}$ shows up due to the following argument. If $C$ is a $\mathbb{P}^1$ with non-trivial $\mathbb{Z}_2$-isotropy at $0$ and $\infty$, and $f: C \rightarrow [\mathbb{P}^1/\mathbb{Z}_2]$ is a map of degree 1, then the orbifold Riemann-Roch theorem says that $h^1(C, f^*\mathcal{O}_{\mathbb{P}^1}(-1)) = 1$. Therefore, it only remains to compute the torus weights of this line bundle. One way to do this is to use the perspective of admissible covers: the data of $f$ is equivalent to the data of a hyperelliptic covering $\widetilde{f}: \mathbb{P}^1 \rightarrow C$, branched at $0$ and $\infty$. It turns out that $H^1(C, f^*\mathcal{O}_{\mathbb{P}^1}(-1)) = H^1(\mathbb{P}^1, \widetilde{f}^*(f^*\mathcal{O}_{\mathbb{P}^1}(-1)))$, and from this, one can derive that the torus weight is $(a + \frac{1}{2})$, see (\cite{RenzoThesis}, pg. 22) for a complete derivation. 

The sequences (\ref{exactuntwisted}) and (\ref{exacttwisted}) show up prominently when we compute $\mathbb{Z}_2$ Hurwitz-Hodge integrals that have no $\psi$-classes i.e. pure Hodge integrals.


\subsection{Non-Pure Hodge Integrals}

In this section, we compute two families of auxiliary integrals that provide recursions for $\mathbb{Z}_2$ Hurwitz-Hodge integrals with at least one $\psi$-class at one marked point. 

For the sake of readability, we make the following definitions:

\begin{Definition} The following is shorthand notation for symbols with repeated use:
\begin{itemize}
\item{$\vec{i} := (i_1, i_2, \ldots, i_n)$} 
\item{For $\vec{\ell} := (\ell_1, \ldots, \ell_n)$, we say $\vec{\ell} \leq \vec{i} \iff \ell_j \leq i_j$ for all $j$}
\item{$|\vec{i}| = i_1 + \ldots + i_n$}
\item{$\lambda_{\vec{i}} := \lambda_{i_1}\ldots\lambda_{{i_n}}$}
\item{$D_{(\vec{i}), 2g + 2} := \int_{\mbar_{0, (2g + 2)t}(\bztwo)}(\psi_j)^{2g - 1 - |\vec{i}|}\lambda_{\vec{i}}$}
\item{$d_{(\vec{i}), 2g + 2} := \int_{\mbar_{0, (2g + 2)t, 1u}(\bztwo)}(\psi_{2g + 3})^{2g - |\vec{i}|}\lambda_{\vec{i}}$}
\end{itemize}
\end{Definition}

\noindent For a fixed $\vec{i}$, and genus $g > 0$, let $k$ be an integer such that $0 \leq k \leq 2g - 2 - |\vec{i}|$. We define

\begin{equation*}
I_k := \left(\int_{\mbar_{0, (2g + 2)t}([\mathbb{P}^1/\mathbb{Z}_2], 1)}\lambda_{\vec{i}}\left(\prod_{i = 1}^{2 + k}\text{ev}_i^*(0)\right)\text{ev}_{2g + 2}^*(\infty)\right)\vert_{t = 1} = 0
\end{equation*}

\begin{Remark}
There is a slight abuse of notation. We use $t$ to refer to the equivariant parameter of the Chow ring of $\aux$, but we we also use $t$ as a decoration on marked points to indicate non-trivial isotropy. However, the contexts in which these two uses occur are distinct and discernible, so no confusion should arise.
\end{Remark}

\noindent We establish notation that enumerates which localization graphs contribute to $I_k$. 

\begin{Definition}
Define $\Gamma_j$ as the set of all localization graphs with the following properties:
\begin{enumerate}
\item{The first $2 + k$ marked points map to $0$}
\item{The last marked point maps to $\infty$}
\item{There are $j$ marked points, distinct from the first $2 + k$ marked points and the last marked point, that map to $\infty$}
\item{The remaining $2g - 1 - k - j$ marked points map to $0$}
\end{enumerate} 
\end{Definition}

\noindent In order to localize $I_k$, we need to compute the contributions coming from all of the localization graphs in $\Gamma_j$ for $0 \leq j \leq 2g - 1 - k$. If we let $\alpha$ be the integrand in $I_k$, and if we define $\text{Cont}(\Gamma_j) := \sum_{\Gamma \in \Gamma_j}\int_{\mbar_{\Gamma}}\frac{\alpha\vert_{\Gamma}}{e(N_{\Gamma})}$, we see that 

\begin{equation*}
I_k = \sum_{j = 0}^{2g - 1 - k}\text{Cont}(\Gamma_j)\vert_{t = 1}
\end{equation*}

\noindent First, consider $\Gamma_0$. This set contains only one localization graph $\Gamma$, and its corresponding moduli space is

\begin{equation*}
\mbar_{\Gamma} = \mbar_{0, (2g + 2)t}(\bztwo)
\end{equation*} 

\noindent Since $\frac{1}{e(N_\Gamma)} = \frac{-1}{t^2(t - \psi_{v_0})}$, the contribution coming from $\Gamma$ is

\begin{equation*}
(-1)\int_{\mbar_{0, (2g + 2)t}(\bztwo)}\frac{-\lambda_{\vec{i}}}{1 - \psi_j}
\end{equation*}

\noindent and therefore,

\begin{equation*}
\text{Cont}(\Gamma_0)\vert_{t = 1} = D_{(\vec{i}), 2g + 2}
\end{equation*}

\noindent Now consider $\Gamma_j$ where $j$ is odd and $1 \leq j \leq 2 \left \lfloor \frac{2g - 2 - k}{2} \right \rfloor + 1$. Let $g_1 := \frac{2g - 1 - j}{2}$ and $g_2 := \frac{j - 1}{2}$. For all $\Gamma \in \Gamma_j$, the corresponding moduli space is

\begin{equation*}
\mbar_\Gamma = \mbar_{0, (2g_1 + 2)t, 1u}(\bztwo) \times \mbar_{0, (2g_2 + 2)t, 1u}(\bztwo)
\end{equation*}

\noindent Since $\frac{1}{e(N_\Gamma)} = 2\left(\frac{-1}{t^2(t - \psi_{v_0})}\right) \times \left(\frac{1}{-t - \psi_{v_\infty}}\right)$, we have

\begin{align*}
\text{Cont}(\Gamma)\vert_{t = 1} & = 2{2g - 1 - k \choose j}(-1)\sum_{\vec{\ell}_1 + \vec{\ell}_2 = \vec{i}}\int_{\mbar_{0, (2g_1 + 2)t, 1u}(\bztwo)}\frac{\lambda_{\vec{\ell}_1}}{1 - \psi_{2g_1 + 3}}\int_{\mbar_{0, (2g_2 + 2)t, 1u}}\frac{\lambda_{\vec{\ell}_2}}{1 + \psi_{2g_2 + 3}} \\
& = 2{2g - 1 - k \choose 2g_2 + 1}(-1)\sum_{\vec{\ell}_1 + \vec{\ell}_2 = \vec{i}}(-1)^{2g_2 - |\vec{\ell}_2|}d_{\vec{\ell}_1, 2g_1 + 2}d_{\vec{\ell}_2, 2g_2 + 2} \\
& = -2{2g - 1 - k \choose 2g_2 + 1}\sum_{\vec{\ell}_2 \leq \vec{i}}(-1)^{|\vec{\ell}_2|}d_{(\vec{i} - \vec{\ell}_2), 2g_1 + 2}d_{\vec{\ell}_2, 2g_2 + 2}
\end{align*}

\noindent and therefore,

\begin{equation*}
\sum_{\substack{j \ \text{odd} \\ 1 \leq j \leq 2 \left \lfloor \frac{2g - 2 - k}{2} \right \rfloor + 1}} \text{Cont}(\Gamma_j)\vert_{t = 1} = -2\sum_{\substack{g_1 + g_2 = g - 1 \\ 0 \leq g_2 \leq \left \lfloor \frac{2g - 2 - k}{2} \right \rfloor \\ \vec{\ell}_2 \leq \vec{i}}}(-1)^{|\vec{\ell}_2|}{2g - 1 - k \choose 2g_2 + 1}d_{(\vec{i} - \vec{\ell}_2), 2g_1 + 2}d_{\vec{\ell}_2, 2g_2 + 2}
\end{equation*}

\noindent Next, we consider $\Gamma_j$ where $j$ is even, and $2 \leq j \leq 2 \left \lfloor \frac{2g - 1 - k}{2} \right \rfloor$. Let $g_1 := \frac{2g - j}{2}$ and $g_2 := \frac{j}{2}$. For each $\Gamma \in \Gamma_j$, the corresponding moduli space is

\begin{equation*}
\mbar_\Gamma = \mbar_{0, (2g_1 + 2)t}(\bztwo) \times \mbar_{0, (2g_2 + 2)t}(\bztwo)
\end{equation*}

\noindent Since $\frac{1}{e(N_\Gamma)} = 2\left(\frac{-1}{t^2(t - \psi_{v_0}}\right) \times \left(\frac{1}{-t - \psi_{\infty}}\right)$, we have

\begin{align*}
\text{Cont}(\Gamma)\vert_{t = 1} & = (2){2g - 1 - k \choose j}(-1)\sum_{\vec{\ell}_1 + \vec{\ell}_2 = \vec{i}}\int_{\mbar_{0, (2g_1 + 2)t}(\bztwo)}\frac{\lambda_{\vec{\ell}_1}}{1 - \psi_j}\int_{\mbar_{0, 2g_2 + 2}(\bztwo)}\frac{\lambda_{\vec{\ell}_2}}{1 + \psi_j} \\
& = (2){2g - 1 - k \choose 2g_2}(-1)\sum_{\vec{\ell}_1 + \vec{\ell}_2 = \vec{i}}(-1)^{2g_2 - 1 - |\vec{\ell}_2|}D_{\vec{\ell}_1, 2g_1 + 2}D_{\vec{\ell}_2, 2g_2 + 2}
\end{align*}

\noindent and therefore,

\begin{equation*}
\sum_{\substack{j \ \text{even} \\ 2 \leq j \leq 2 \left \lfloor \frac{2g - 1 - k}{2} \right \rfloor}} \text{Cont}(\Gamma_j)\vert_{t = 1} = 2\sum_{\substack{g_1 + g_2 = g \\ 1 \leq g_2 \leq \left \lfloor \frac{2g - 1 - k}{2} \right \rfloor \\ \vec{\ell} \leq \vec{i}}}(-1)^{|\vec{\ell}|}{2g - 1 - k \choose 2g_2}D_{(\vec{i} - \vec{\ell}), 2g_1 + 2}D_{\vec{\ell}, 2g_2 + 2}
\end{equation*}

\noindent Now we consider a different auxiliary integral. For fixed $\vec{i}$ and genus $g > 0$,  let $0 \leq k \leq 2g - 1 - |\vec{i}|$ and define

\begin{equation*}
\widetilde{I}_k := \left(\int_{\mbar_{0, (2g + 2)t}([\mathbb{P}^1/\mathbb{Z}_2], 1)}\lambda_{\vec{i}}\left(\prod_{i = 1}^{2 + k}\text{ev}_i^*(0)\right)\right)\vert_{t = 1} = 0
\end{equation*}

\noindent As we saw in the case of $I_k$, we first need to enumerate the localization graphs that contribute to $\widetilde{I}_k$.

\begin{Definition}
Define $\widetilde{\Gamma}_j$ as the set of all localization graphs of $\mbar_{0, (2g + 2)t}([\mathbb{P}^1/\mathbb{Z}_2], 1)$ with the following properties:

\begin{enumerate}
\item{The first $2 + k$ marked points map to 0}
\item{There are $j$ marked points, distinct from the first $2 + k$ marked points, that map to $\infty$.}
\item{The remaining $2g - k - j$ marked points map to $0$.}
\end{enumerate}
\end{Definition}

\noindent Let $\alpha$ be the integrand appearing in $\widetilde{I}_k$, and define $\text{Cont}(\widetilde{\Gamma}_j) := \sum_{\Gamma \in \Gamma_j^B}\int_{\mbar_\Gamma}\frac{\alpha\vert_{\Gamma}}{e(N_{\Gamma})}$. When we localize $\widetilde{I}_k$, we need to compute the contributions coming from the sets $\widetilde{\Gamma}_j$,

\begin{equation*}
\widetilde{I}_k = \sum_{j = 0}^{|\vec{i}| + 1}\text{Cont}(\widetilde{\Gamma}_j)\vert_{t = 1}
\end{equation*}

\noindent The computations for $\text{Cont}(\widetilde{\Gamma}_j)\vert_{t = 1}$ are analagous to the computations made in the case of $I_k$. We won't go through all of the details, but instead, simply state the results of the computation:

\begin{equation*}
\text{Cont}(\widetilde{\Gamma}_0)\vert_{t = 1} = d_{\vec{i}, 2g + 2}
\end{equation*}

\begin{equation*}
\sum_{\substack{j \ \text{odd} \\ 1 \leq j \leq 2\left\lfloor\frac{2g - 1 - k}{2}\right\rfloor + 1}}\text{Cont}(\widetilde{\Gamma}_j)\vert_{t = 1} = -2\sum_{\substack{g_1 + g_2 = g \\ 0 \leq g_2 \leq \left\lfloor\frac{2g - 1 - k}{2}\right\rfloor \\ \vec{\ell} \leq \vec{i}}}(-1)^{|\vec{\ell}|}{2g - k \choose 2g_2 + 1}D_{(\vec{i} - \vec{\ell}), 2g_1 + 2}D_{\vec{\ell}, 2g_2 + 2}
\end{equation*}

\begin{equation*}
\sum_{\substack{j \ \text{even} \\ 2 \leq j \leq 2\left\lfloor\frac{2g - k}{2}\right\rfloor}}\text{Cont}(\widetilde{\Gamma}_j)\vert_{t = 1} = 2\sum_{\substack{g_1 + g_2 = g - 1 \\ 0 \leq g_2 \leq \left\lfloor\frac{2g - k}{2} \right\rfloor - 1 \\ \vec{\ell} \leq \vec{i}}}(-1)^{|\vec{\ell}|}{2g - k \choose 2g_2 + 2}d_{(\vec{i} - \vec{\ell}), 2g_1 + 2}d_{\vec{\ell}, 2g_2 + 2}
\end{equation*}

\noindent Combining all of the previous localization computations, we obtain the following theorem:


\begin{Theorem}\label{NPHrecursion}
Fix $\vec{i}$ and $g > 0$ such that $|\vec{i}| < 2g - 1$. For $0 \leq k \leq 2g - 2 - |\vec{i}|$,

\begin{align}
D_{\vec{i}, 2g + 2} = 2 & \sum_{\substack{g_1 + g_2 = g - 1 \\ 0 \leq g_2 \leq \left\lfloor\frac{2g - 2 - k}{2}\right\rfloor \\ \vec{\ell} \leq \vec{i}}}(-1)^{|\vec{\ell}|}{2g - 1 - k \choose 2g_2 + 1}d_{(\vec{i} - \vec{\ell}), 2g_1 + 2}d_{\vec{\ell}, 2g_2 + 2} \label{NPHtwisted} \\
& - 2 \sum_{\substack{g_1 + g_2 = g \\ 1 \leq g_2 \leq \left\lfloor\frac{2g - 1 - k}{2}\right\rfloor \\ \vec{\ell} \leq \vec{i}}}(-1)^{|\vec{\ell}|}{2g - 1 - k \choose 2g_2}D_{(\vec{i} - \vec{\ell}), 2g_1 + 2}D_{\vec{\ell}, 2g_2 + 2} \nonumber
\end{align}

\noindent and for $0 \leq k \leq 2g - 1 - |\vec{i}|$,

\begin{align}
d_{\vec{i}, 2g + 2} = 2 & \sum_{\substack{g_1 + g_2 = g \\ 0 \leq g_2 \leq \left\lfloor\frac{2g - 1 - k}{2}\right\rfloor \\ \vec{\ell} \leq \vec{i}}} (-1)^{|\vec{\ell}|}{2g - k \choose 2g_2 + 1}D_{(\vec{i} - \vec{\ell}), 2g_1 + 2}D_{\vec{\ell}, 2g_2 + 2} \label{NPHuntwisted} \\
& - 2 \sum_{\substack{g_1 + g_2 = g - 1 \\ 0 \leq g_2 \leq \left\lfloor\frac{2g - k}{2}\right\rfloor - 1 \\ \vec{\ell} \leq \vec{i}}} (-1)^{|\vec{\ell}|}{2g - k \choose 2g_2 + 2}d_{(\vec{i} - \vec{\ell}), 2g_1 + 2}d_{\vec{\ell}, 2g_2 + 2} \nonumber
\end{align}

\end{Theorem}


\subsection{Pure Hodge Integrals}

The only integrals in which the above recursion does not apply is the case where $|\vec{i}| = 2g - 1$ (i.e. when there are no $\psi$-classes in the integrand) over the space $\mbar_{0, (2g + 2)t}(\bztwo)$. In order to form a full set of recursions, we compute a new auxiliary integral on $\aux$. 

Before we begin the localization computation, we establish the following notation:

\begin{Definition}\label{equivariantlinebundle}
Define $L_a$ to be the $\mathbb{C}^*$-equivariant trivial line bundle on $\aux$ with torus weight $a$. In other words, $\mathbb{C}^*$ acts on each fiber by $\lambda \cdot z = \lambda^a z$.
\end{Definition}

\noindent Notice that $c^{\mathbb{C}^*}(L_a) = 1 + at$. In our localization computations below, we make repeated use of the following fact concerning Chern classes (see \cite{Fulton}):

\begin{Proposition}
Let $\mathbb{E}$ be a rank $r$ vector bundle, and let $L$ be a line bundle. Then

\begin{equation*}
c_i(E \otimes L) = \sum_{j = 0}^i{r - i + j \choose j}c_{i - j}(E)c_1(L)^j
\end{equation*}
\end{Proposition}

\noindent Let $\vec{i} = (i_1, \ldots, i_{n}) \in \mathbb{Z}^n_{\geq 0}$ such that $|\vec{i}| = 2g - 1$. Without loss of generality, assume $i_n > 0$. For notational convenience, we define $\vec{m} := (i_1, \ldots, i_{n - 1})$, the tuple $\vec{i}$ that is missing the last entry. Consider the auxiliary integral

\begin{equation*}
I := \left(\int_{\aux}\lambda_{\vec{m}} \cdot c_{i_n + 1}(R^i\pi_*f^*\mathcal{O}_{\mathbb{P}^1}(-1)) \cdot \text{ev}_1^*(\infty)\right)\vert_{t = 1} = 0
\end{equation*}

\noindent where we choose the torus weights of $\mathcal{O}_{\mathbb{P}^1}(-1)$ to be $-1$ over $0$ and $0$ over $\infty$. Expressing $I$ as a sum of integrals over fixed loci results in recursions for the integrals 

\begin{equation*}
D_{(\vec{i}), 2g + 2} := \int_{\mbar_{0, (2g + 2)t}(\bztwo)}\lambda_{\vec{i}}
\end{equation*}

\noindent In order to enumerate the localization graphs that contribute to $I$, we make the following definition:

\begin{Definition}
Let $\Gamma_i$ denote the set of all localization graphs for $\aux$ with the following properties:

\begin{enumerate}
\item{The first marked point lies over $\infty$}
\item{There are $i$ marked points, distinct from the first marked point, that lie over $\infty$}
\item{The remaining $2g + 1 - i$ points lie over $0$} 
\end{enumerate}
\end{Definition}

\noindent In order to compute $I$, we compute the contributions coming from the sets $\Gamma_i$, as $i$ ranges from $0$ to $2g + 1$. Let $\alpha \in A^*(\aux)$ represent the integrand in $I$. If we define $\text{Cont}(\Gamma_i) := \sum_{\Gamma \in \Gamma_i}\int_{\mbar_{\Gamma}}\frac{\alpha\vert_{\Gamma}}{e(N_{\Gamma})}$ then

\begin{equation*}
I = \sum_{i = 0}^{2g + 1}\text{Cont}(\Gamma_i)\vert_{t = 1}
\end{equation*} 

\noindent Since $|\Gamma_i| = {2g + 1 \choose i}$, and since each contribution is the same for each $\Gamma \in \Gamma_i$, when we compute $\text{Cont}(\Gamma_i)\vert_{t = 1}$, we pick up a factor of ${2g + 1 \choose i}$. \\

\noindent First, we show that

\begin{equation*}
\text{Cont}(\Gamma_{2g - 1})\vert_{t = 1} = \text{Cont}(\Gamma_{2g + 1})\vert_{t = 1} = 0
\end{equation*}

\noindent Consider the set $\Gamma_{2g - 1}$. Each localization graph $\Gamma \in \Gamma_{2g - 1}$ corresponds to the moduli space

\begin{equation*}
\mbar_{\Gamma} = \mbar_{0, 2t, 1u}(\bztwo) \times \mbar_{0, (2(g - 1) + 2)t, 1u}(\bztwo)
\end{equation*}

\noindent Applying cohomology to the normalization exact sequence, we get

\begin{center}
\begin{tikzcd}
0 \arrow[r] & L_{-1} \oplus L_0 \arrow[r] & H^1(C, f^*(\mathcal{O}_{\mathbb{P}^1}(-1))) \arrow[r] & H^1(C_{v_\infty}, \mathbb{C}\times C_{v_\infty}) \arrow[r] & 0
\end{tikzcd}
\end{center}

\noindent so for any graph $\Gamma \in \Gamma_{2g - 1}$, we have

\begin{align*}
c_{i_n + 1}(R^1(\pi_*f^*\mathcal{O}(-1)))\vert_{\mbar_{\Gamma}, t = 1} & = c_{i_n+1}^{\mathbb{C}^*}\left(L_{-1} \oplus L_{0} \oplus \left(\mathbb{E}_{g - 1}^\vee \otimes L_0\right)\right)\vert_{t = 1} \\
& = c_{i_n + 1}^{\mathbb{C}^*}(L_{-1} \oplus (\mathbb{E}_{g - 1}^\vee \otimes L_0))\vert_{t = 1} \\ 
& = c_{i_n + 1}^{\mathbb{C}^*}\left(\mathbb{E}_{g - 1}^\vee \otimes L_0\right) - c_{i_n}^{\mathbb{C}^*}\left(\mathbb{E}_{g - 1}^\vee \otimes L_0\right)\vert_{t = 1} \\
& = \sum_{r = 0}^{i_n + 1}{g - 1 - i_n - 1 + r \choose r}c_{i_n + 1 - r}\left(\mathbb{E}_{g - 1}^\vee\right)(0)^r  \\
& - \sum_{r = 0}^{i_n}{g - 1 - i_n + r \choose r}c_{i_n - r}\left(\mathbb{E}_{g - 1}^\vee\right)(0)^r \\
& = (-1)^{i_n + 1}\lambda_{i_n + 1} - (-1)^{i_n}\lambda_{i_n} \\
& = (-1)^{i_n + 1}(\lambda_{i_n + 1} + \lambda_{i_n})
\end{align*}

\noindent For each $\Gamma \in \Gamma_{2g - 1}$, $\frac{1}{e(N_{\Gamma})} = \left(\frac{-1}{t^2(t - \psi)}\right) \times \left(\frac{1}{-t - \psi}\right)$. Therefore, we have

\begin{align*}
\text{Cont}(\Gamma_{2g - 1})\vert_{t = 1} & = {2g + 1 \choose 2g - 1}(-1)\int_{\mbar_{0. 2t, 1u}(\bztwo)}\frac{-1}{1 - \psi_3}\int_{\mbar_{0, (2(g - 1) + 2)t, 1u}(\bztwo)}\frac{(-1)^{i_n + 1}\lambda_{i_n + 1}\lambda_{\vec{m}} + (-1)^{i_n + 1}\lambda_{i_n}\lambda_{\vec{m}}}{-1 - \psi_{2g + 1}} \\
& = {2g + 1 \choose 2g - 1}(-1)^{i_n}\int_{\mbar_{0, (2(g - 1) + 2)t}(\bztwo)}\frac{\lambda_{\vec{m}}\lambda_{i_n + 1} + \lambda_{\vec{m}}\lambda_{i_n}}{1 + \psi_{2g + 1}} \\
& = 0
\end{align*}

\noindent where the last equality holds for dimension reasons, since $\text{dim}\left(\mbar_{0, (2(g - 1) + 2)t}(\bztwo)\right) = 2g - 2$. \\

Now consider the set $\Gamma_{2g + 1}$. Each localization graph $\Gamma \in \Gamma_{2g + 1}$ corresponds to the moduli space

\begin{equation*}
\mbar_{\Gamma} = \mbar_{0, (2g + 2)t, 1u}
\end{equation*}

\noindent Applying cohomology to the normalization long exact sequence, we get

\begin{center}
\begin{tikzcd}
0 \arrow[r] & L_0 \arrow[r] & H^1(C, f^*\mathcal{O}_{\mathbb{P}^1}(-1)) \arrow[r] & H^1(C_{v_{\infty}}, \mathbb{C} \times C_{v_{\infty}}) \arrow[r] & 0
\end{tikzcd}
\end{center}

\noindent Therefore, we have

\begin{align*}
c_{i_n + 1}(R^1\pi_*f^*\mathcal{O}_{\mathbb{P}^1}(-1))\vert_{\mbar_{\Gamma}, t = 1} & = c_{i_n + 1}^{\mathbb{C}^*}(L_0 \oplus (\mathbb{E}_g^\vee \otimes L_0))\vert_{t = 1} \\
& = c_{i_n + 1}^{\mathbb{C}^*}(\mathbb{E}_g^\vee \otimes L_0)\vert_{t = 1} \\
& = \sum_{r = 0}^{i_n + 1}{g - i_n - 1 + r \choose r}c_{i_n + 1 - r}(\mathbb{E}_g^\vee)(0)^r \\
& = (-1)^{i + 1}\lambda_{i + 1}
\end{align*}

Since $\frac{1}{e(N_{\Gamma})} = \frac{-1}{t^2(-t - \psi_{v_{\infty}}}$,

\begin{equation*}
\text{Cont}(\Gamma_{2g + 1})\vert_{t = 1} = (-1)\int_{\mbar_{0, (2g + 2)t, 1u}(\bztwo)}\frac{(-1)^{i_n + 1}\lambda_{i_n + 1}\lambda_{\vec{m}}}{-1 - \psi_{2g + 3}} = 0
\end{equation*}

\noindent The reason for the above vanishing is the following: in the integrand above we have a monomial of $\lambda$-classes whose total degree is $2g$. Every monomial of $\lambda$ classes on the space $\mbar_{0, (2g + 2)t, 1u}(\bztwo)$ is pulled back from $\mbar_{0, (2g +2)t}$ (along the map that forgets the untwisted point). But since $\text{dim}(\mbar_{0, (2g + 2)t}(\bztwo) = 2g - 1$, this monomial is zero in Chow. 

For the remainder of this localization computation, we organize the workflow in the following way. First, we compute $\text{Cont}(\Gamma_0)\vert_{t = 1}, \text{Cont}(\Gamma_1)\vert_{t = 1}, \text{Cont}(\Gamma_{2g})\vert_{t = 1}$; these contributions are special in that the corresponding moduli spaces are not products. We split the remaining cases into two situations: $\left\{\text{Cont}(\Gamma_j)\vert_{t = 1}\right\}_{2, \leq j \leq 2g - 2, j \ \text{even}}$, and $\left\{\text{Cont}(\Gamma_j)\vert_{t = 1}\right\}_{3 \leq j \leq 2g - 3, j \ \text{odd}}$. In the former case, the corresponding moduli spaces are products of spaces with no untwisted points, and the latter case corresponds to spaces with one untwisted point.

Lets begin with $\text{Cont}(\Gamma_0)$. There is only graph $\Gamma \in \Gamma_0$, and the corresponding moduli space is 

\begin{equation*}
\mbar_{\Gamma} = \mbar_{0, (2g + 2)t}(\bztwo)
\end{equation*}

\noindent Applying cohomology to the normalization long exact sequence, we get

\begin{center}
\begin{tikzcd}
0 \arrow[r] & H^1(C, f^*\mathcal{O}_{\mathbb{P}^1}(-1)) \arrow[r] & H^1(C_{v_0}, \mathbb{C} \times C_{v_0}) \oplus H^1(C_e, f^*\mathcal{O}_{\mathbb{P}^1}(-1)) \arrow[r] & 0
\end{tikzcd}
\end{center}

\noindent and therefore,

\begin{align*}
c_{i_n + 1}(R^1\pi_*f^*\mathcal{O}_{\mathbb{P}^1}(-1))\vert_{\mbar_{\Gamma}, t = 1} & = c_{i_n + 1}^{\mathbb{C}^*}((\mathbb{E}_g^\vee \otimes L_{-1}) \oplus L_{-\frac{1}{2}})\vert_{t = 1} \\
& = c_{i_n + 1}^{\mathbb{C}^*}(\mathbb{E}_g^\vee \otimes L_{-1}) - \frac{1}{2}tc_{i_n}^{\mathbb{C}*}(\mathbb{E}_g^\vee \otimes L_{-1})\vert_{t = 1} \\
& = \sum_{r = 0}^{i_n + 1}{g - i_n - 1 + r \choose r}c_{i_n + 1 - r}(\mathbb{E}_g^\vee)(-1)^r - \frac{1}{2}\sum_{r = 0}^{i_n}{g - i_n + r \choose r} c_{i_n - r}(\mathbb{E}_g^\vee)(-1)^r \\
& = \sum_{r = 0}^{i_n + 1}{g - i_n - 1 + r \choose r}(-1)^{i_n + 1}\lambda_{i_n + 1 - r} - \frac{1}{2}\sum_{r = 0}^{i_n}{g - i_n + r \choose r}(-1)^{i_n}\lambda_{i_n - r} \\ 
& = (-1)^{i_n + 1}\left[ \sum_{r = 0}^{i_n + 1}{g - i_n - 1 + r \choose r}\lambda_{i_n + 1 - r} + \frac{1}{2}\sum_{r = 0}^{i_n}{g - i_n + r \choose r}\lambda_{i_n - r} \right]
\end{align*}

\noindent Since $\frac{1}{e(N_{\Gamma})} = \frac{-1}{t^2(t - \psi_{v_0})}$, we have

\begin{align*}
\text{Cont}(\Gamma_0)\vert_{t = 1} = (-1) \int_{\mbar_{0, (2g + 2)t}(\bztwo)}\frac{(-1)(-1)^{i_n + 1}\left[\sum_{r = 0}^{i_n + 1}{g - i_n - 1 + r \choose r}\lambda_{\vec{m}}\lambda_{i_n + 1 - r} + \frac{1}{2}\sum_{r = 0}^{i_n}{g - i_n + r \choose r}\lambda_{\vec{m}}\lambda_{i_n - r}\right]}{1 - \psi_1}
\end{align*}

\noindent After some simplification, we obtain:

\begin{align*}
\text{Cont}(\Gamma_0)\vert_{t = 1} = (-1)^{i_n + 1} & \left(g - i_n + \frac{1}{2}\right)D_{(\vec{i}), 2g + 2} \\
& + (-1)^{i_n + 1}\sum_{r = 1}^{i_n}\left[{g - i_n + r \choose r + 1} + \frac{1}{2}{g - i_n + r \choose r}\right]D_{(\vec{m}, i_n - r), 2g + 2}
\end{align*}

\noindent Now consider $\Gamma_1$. For all $\Gamma \in \Gamma_1$, we have

\begin{equation*}
\mbar_{\Gamma} = \mbar_{0, (2(g - 1) + 2)t, 1u}(\bztwo) \times \mbar_{0, 2t, 1u}(\bztwo)
\end{equation*}

\noindent Applying cohomology to the normalization long exact sequence, we get

\begin{center}
\begin{tikzcd}
0 \arrow[r] & L_{-1} \oplus L_0 \arrow[r] & H^1(C, f^*\mathcal{O}_{\mathbb{P}^1}(-1)) \arrow[r] & H^1(C_{v_0}, \mathbb{C} \times C_{v_0}) \arrow[r] & 0
\end{tikzcd}
\end{center}

\noindent Therefore, we have

\begin{align*}
c_{i_n + 1}(R^1\pi_*f^*\mathcal{O}_{\mathbb{P}^1}(-1))\vert_{\mbar_{\Gamma}, t = 1} & = c_{i_n + 1}^{\mathbb{C}^*}(L_{-1} \oplus L_0 \oplus (\mathbb{E}_{g - 1}^\vee \otimes L_{-1}))\vert_{t = 1} \\
& = c_{i_n + 1}^{\mathbb{C}^*}(L_{-1} \oplus (\mathbb{E}_{g - 1}^\vee \otimes L_{-1}))\vert_{t = 1} \\
& = c_{i_n + 1}^{\mathbb{C}^*}(\mathbb{E}_{g - 1} \otimes L_{-1})\vert_{t = 1} - tc_{i_n}^{\mathbb{C}^*}(\mathbb{E}_{g - 1}^\vee \otimes L_{-1})\vert_{t = 1} \\
& = \sum_{r = 0}^{i_n + 1}{g - 1 - i_n - 1 + r \choose r}c_{i_n + 1 - r}(\mathbb{E}_{g - 1}^\vee)(-1)^r \\
& - \sum_{r = 0}^{i_n}{g - 1 - i_n + r \choose r}c_{i_n - r}(\mathbb{E}_{g - 1}^\vee)(-1)^r \\
& = (-1)^{i_n + 1}\left[\sum_{r = 0}^{i_n + 1}{g - 2 - i_n + r \choose r}\lambda_{i_n + 1 - r} + \sum_{r = 0}^i{g - 1 - i_n + r \choose r}\lambda_{i_n - r}\right]
\end{align*}

\noindent Since $\frac{1}{e(N_{\Gamma})} = \left(\frac{-1}{t^2(t - \psi_{v_0})}\right) \times \left(\frac{1}{-t - \psi_{v_{\infty}}}\right)$, we have

\begin{align*}
{2g + 1 \choose 1}(-1) & \int_{\mbar_{0, (2(g - 1) + 2)t, 1u}(\bztwo)}\frac{(-1)(-1)^{i_n + 1}\left[\sum_{r = 0}^{i_n + 1}{g - 2 - i_n + r \choose r} \lambda_{\vec{m}}\lambda_{i_n + 1 - r} + \sum_{r = 0}^{i_n}{g - 1 - i_n + r \choose r}\lambda_{\vec{m}}\lambda_{i_n - r}\right]}{1 - \psi_{2g + 1}} \\
& \times \int_{\mbar_{0, 2t, 1u}(\bztwo)}\frac{1}{-1 - \psi_3}
\end{align*}

\noindent After some simplification, we get:

\begin{equation*}
\text{Cont}(\Gamma_1)\vert_{t = 1} = (2g + 1)(-1)^{i_n}\left[\sum_{r = 1}^{i_n}{g - i_n + r \choose r + 1} d_{(\vec{m}, i_n - r), 2(g - 1) + 2}\right]
\end{equation*}

\noindent Now consider $\Gamma_{2g}$. For all $\Gamma \in \Gamma_{2g}$, we have

\begin{equation*}
\mbar_{\Gamma} = \mbar_{0, (2g + 2)t}(\bztwo)
\end{equation*}

\noindent Applying cohomology to the normalization long exact sequence, we get

\begin{center}
\begin{tikzcd}
0 \arrow[r] & H^1(C, f^*\mathcal{O}_{\mathbb{P}^1}(-1)) \arrow[r] & H^1(C_e, f^*\mathcal{O}_{\mathbb{P}^1}) \oplus H^1(C_{v_{\infty}}, \mathbb{C} \times C_{v_{\infty}}) \arrow[r] & 0
\end{tikzcd}
\end{center}

\noindent and therefore, 

\begin{align*}
c_{i_n + 1}(R^1\pi_*f^*\mathcal{O}_{\mathbb{P}^1}(-1))\vert_{\mbar_{\Gamma}, t = 1} & = c_{i_n + 1}^{\mathbb{C}^*}(L_{\frac{-1}{2}} \oplus (\mathbb{E}_g^\vee \otimes L_0))\vert_{t = 1} \\
& = c_{i_n + 1}^{\mathbb{C}^*}(\mathbb{E}_g^\vee \otimes L_0)\vert_{t = 1} - \frac{1}{2}tc_{i_n}^{\mathbb{C}^*}(\mathbb{E}_g^\vee \otimes L_0)\vert_{t = 1} \\
& = \sum_{r = 0}^{i_n + 1}{g - i_n - 1 + r \choose r}c_{i_n + 1 - r}(\mathbb{E}_g^\vee)(0)^r - \frac{1}{2}\sum_{r = 0}^{i_n}{g - i_n + r \choose r}c_{i_n - r}(\mathbb{E}_g^\vee)(0)^r \\
& = (-1)^{i_n + 1}\left[\lambda_{i_n + 1} + \frac{1}{2}\lambda_{i_n}\right]
\end{align*}

\noindent Since $\frac{1}{e(N_{\Gamma})} = \frac{-1}{t^2(t - \psi_{v_{\infty}}}$, it follows that

\begin{equation*}
\text{Cont}(\Gamma_{2g})\vert_{t = 1} = {2g + 1 \choose 2g}(-1)\int_{\mbar_{0, (2g + 2)t}(\bztwo)}\frac{(-1)(-1)^{i_n + 1}\left[\lambda_{\vec{m}}\lambda_{i_n + 1} + \frac{1}{2}\lambda_{\vec{m}}\lambda_{i_n}\right]}{-1 - \psi_1}
\end{equation*}

\noindent After some simplification, we have

\begin{equation*}
\text{Cont}(\Gamma_{2g})\vert_{t = 1} = (2g + 1)(-1)^{i_n}\left(\frac{1}{2}\right)D_{(\vec{i}), 2g + 2}
\end{equation*}

\noindent Now consider $\Gamma_j$, where $2 \leq j \leq 2g - 2$, and $j$ is even. First, we make the variable substitutions $g_1 := \frac{2g - j}{2}$ and $g_2 := \frac{j}{2}$, so that for $\Gamma \in \Gamma_{j}$, we have

\begin{align*}
\mbar_{\Gamma} & = \mbar_{0, (2g + 2 - j)t}(\bztwo) \times \mbar_{0, (j + 2)t}(\bztwo) \\
& = \mbar_{0, (2g_1 + 2)t}(\bztwo) \times \mbar_{0, (2g_2 + 2)t}(\bztwo)
\end{align*}

\noindent Applying cohomology to the normalization long exact sequence, we get

\begin{center}
\begin{tikzcd}
0 \arrow[r] & H^1(C, f^*\mathcal{O}_{\mathbb{P}^1}(-1)) \arrow[r] & H^1(C_e, f^*\mathcal{O}_{\mathbb{P}^1}(-1)) \oplus H^1(C_{v_0}, \mathbb{C} \times C_{v_0}) \oplus H^1(C_{v_{\infty}}, \mathbb{C} \times C_{v_{\infty}}) \arrow[r] & 0
\end{tikzcd}
\end{center}

\noindent Therefore, we have

\begin{align*}
c_{i_n + 1}(R^1\pi_*f^*\mathcal{O}_{\mathbb{P}^1}(-1))\vert_{\mbar_{\Gamma}, t = 1} = c_{i_n + 1}^{\mathbb{C}^*}(L_{\frac{-1}{2}} \oplus (\mathbb{E}_{g_1}^\vee \otimes L_{-1}) \oplus (\mathbb{E}_{g_2}^\vee \otimes L_0))\vert_{t = 1}
\end{align*}

\noindent After a bit of unraveling, the above becomes

\begin{equation*}
(-1)^{i_n + 1}\left[\sum_{p + q = i_n + 1}\left(\sum_{r = 0}^p{g_1 - p + r \choose r}\lambda_{p - r}\right) \times (\lambda_q) + \frac{1}{2}\sum_{p + q = i_n} \left(\sum_{r = 0}^p{g_1 - p + r \choose r}\lambda_{p - r}\right) \times (\lambda_q)\right]
\end{equation*}

\noindent Since $\frac{1}{e(N_{\Gamma})} = (2)\left(\frac{-1}{t^2(t - \psi_{v_0})}\right) \times \left(\frac{1}{-t - \psi_{v_{\infty}}}\right)$,

\begin{align*}
& \text{Cont}(\Gamma_j)\vert_{t = 1} = \\
& {2g + 1 \choose j}(-1)(2)\left[(-1)^{i_n + 1}\sum_{\substack{\vec{\ell}_1 + \vec{\ell}_2 =\vec{m} \\ p + q = i_n + 1}}\int_{\mbar_{0, (2g_1 + 2)t}(\bztwo)}\frac{(-1)\left(\sum_{r = 0}^p{g_1 - p + r \choose r}\lambda_{\vec{\ell}_1}\lambda_{p - r}\right)}{1 - \psi_1}\int_{\mbar_{0, (2g_2 + 2)t}(\bztwo)}\frac{\lambda_{\vec{\ell}_2}\lambda_q}{-1 - \psi_1} \right. \\
& \left. + \frac{1}{2}(-1)^{i_n + 1}\sum_{\substack{\vec{\ell}_1 + \vec{\ell}_2 = \vec{m} \\ p + q = i_n}}\int_{\mbar_{0, (2g_1 + 2)t}(\bztwo)}\frac{(-1)\left(\sum_{r = 0}^p{g_1 - p + r \choose r}\lambda_{\vec{\ell}_1}\lambda_{p - r}\right)}{1 - \psi_1}\int_{\mbar_{0, (2g_2 + 2)t}(\bztwo)}\frac{\lambda_{\vec{\ell}_2}\lambda_q}{-1 - \psi_1}\right]
\end{align*}

\noindent Since summing across all even $j$ such that $2 \leq j \leq 2g - 2$ is equivalent to summing across all $g_1, g_2, g_i > 0$ such that $g_1 + g_2 = g$, we have

\begin{align*}
& \sum_{\substack{2 \leq j \leq 2g - 2 \\ j \ \text{even}}}\text{Cont}(\Gamma_j)\vert_{t = 1} = \\
& (2)(-1)^{i_n}\sum_{\substack{g_1 + g_2 = g, \ g_i > 0 \\ \vec{\ell}_1 + \vec{\ell}_2 = \vec{m} \\ p + q = i_n + 1 \\ 0 \leq r \leq p}} (-1)^{2g_2 - 1- q - |\vec{\ell}_2|}{2g + 1 \choose 2g_2}{g_1 - p + r \choose r}D_{(\vec{\ell}_1, p - r), 2g_1 + 2}D_{(\vec{\ell}_2, q), 2g_2 + 2} \\
& + (-1)^{i_n}\sum_{\substack{g_1 + g_2 = g, \ g_i > 0 \\ \vec{\ell}_1 + \vec{\ell}_2 = \vec{m} \\ p + q = i_n \\ 0 \leq r \leq p}} (-1)^{2g_2 - 1- q - |\vec{\ell}_2|}{2g + 1 \choose 2g_2}{g_1 - p + r \choose r}D_{(\vec{\ell}_1, p - r), 2g_1 + 2}D_{(\vec{\ell}_2, q), 2g_2 + 2}
\end{align*}

Lastly, consider $\Gamma_j$ where $3 \leq j \leq 2g - 3$, and $j$ is odd. At this point, the reader should understand how these contributions are calculated, and therefore, we expedite the exposition by simply stating the results of the computations. For all $\Gamma \in \Gamma_j$, we have

\begin{align*}
c_{i_n + 1}(R^1\pi_*f^*\mathcal{O}_{\mathbb{P}^1}(-1))\vert_{\mbar_{\Gamma}, t = 1} & = (-1)^{i_n + 1}\left[\sum_{p + q = i_n + 1}\left(\sum_{r = 0}^p{g_1 - p + r \choose r}\lambda_{p - r}\right) \times (\lambda_q)\right. \\
& \left. + \sum_{p + q = i_n}\left(\sum_{r = 0}^p{g_1 - p + r \choose r}\lambda_{p - r}\right) \times (\lambda_q) \right]
\end{align*} 

\noindent and therefore,

\begin{align*}
& \sum_{\substack{3 \leq j \leq 2g - 3 \\ j \ \text{odd}}}\text{Cont}(\Gamma_j)\vert_{t = 1} = \\
& (-1)^{i_n}(2)\sum_{\substack{g_1 + g_2 = g - 1, \ g_i > 0 \\ \vec{\ell}_1 + \vec{\ell}_2 = \vec{m} \\ p + q = i_n + 1 \\ 0 \leq r \leq p}} (-1)^{2g_2 - q - |\vec{\ell}_2|}{2g + 1 \choose 2g_2 + 1}{g_1 - p + r \choose r}d_{(\vec{\ell}_1, p - r), 2g_1 + 2}d_{(\vec{\ell}_2, q), 2g_2 + 2} \\
+ & (-1)^{i_n}(2)\sum_{\substack{g_1 + g_2 = g - 1, \ g_i > 0 \\ \vec{\ell}_1 + \vec{\ell}_2 = \vec{m} \\ p + q = i _n\\ 0 \leq r \leq p}} (-1)^{2g_2 - q - |\vec{\ell}_2|}{2g + 1 \choose 2g_2 + 1}{g_1 - p + r \choose r}d_{(\vec{\ell}_1, p - r), 2g_1 + 2}d_{(\vec{\ell}_2, q), 2g_2 + 2}
\end{align*}


\section{Proofs}\label{proofs}

\begin{proof}[Proof of Theorem \ref{Recursion}]
We have recursions for Hodge integrals that have at lease one $\psi$-class insertion (see Theorem \ref{NPHrecursion}). In order to obtain a complete set of recursions that determine all of the intersection numbers $D_{(\vec{i}), 2g + 2}$ and $d_{(\vec{i}), 2g + 2}$, we need a recursion for Hodge integrals $D_{(\vec{i}), 2g + 2}$, where $|\vec{i}| = 2g - 1$. In the previous section, we computed a vanishing localization computation that involved integrals $D_{(\vec{i}), 2g + 2}$ , where $\vec{i} = (i_1, \ldots, i_n), |\vec{i}| = 2g - 1$, and without loss of generality, we assumed $i_n > 0$. It only remains to isolate the term $D_{(\vec{i}), 2g + 2}$, and check that the remaining terms involve either Hodge monomials of lower degree, or Hodge integrals of lower genus. \\

\noindent Since $I = 0$,

\begin{equation*}
\left[\text{Cont}(\Gamma_0) + \text{Cont}(\Gamma_1) + \text{Cont}(\Gamma_{2g}) + \sum_{\substack{2 \leq j \leq 2g - 2 \\ j \ \text{even}}}\text{Cont}(\Gamma_j) + \sum_{\substack{3 \leq j \leq 2g - 3 \\ j \ \text{odd}}} \text{Cont}(\Gamma_j)\right]\vert_{t = 1} = 0
\end{equation*}  

\noindent The only terms that contain $D_{(\vec{i}), 2g + 2}$ are $\text{Cont}(\Gamma_0)$ and $\text{Cont}(\Gamma_{2g})$. and therefore,

\begin{align*}
& (-1)^{i_n + 1}\left(g - i_n + \frac{1}{2}\right)D_{(\vec{i}), 2g + 2} + (2g + 1)(-1)^{i_n}\left(\frac{1}{2}\right)D_{(\vec{i}), 2g + 2}  \\
& + (-1)^{i_n + 1}\sum_{r = 1}^{i_n}\left[{g - i_n + r \choose r + 1} + \frac{1}{2}{g - i_n + r \choose r}\right]D_{(\vec{m}, i_n - r), 2g + 2} \\
& + \left[\text{Cont}(\Gamma_1) + \sum_{\substack{2 \leq j \leq 2g - 2 \\ j \ \text{even}}}\text{Cont}(\Gamma_j) + \sum_{\substack{3 \leq j \leq 2g - 3 \\ j \ \text{odd}}} \text{Cont}(\Gamma_j)\right]\vert_{t = 1} = 0
\end{align*}

\noindent Since

\begin{equation*}
(-1)^{i_n + 1}\left(g - i_n + \frac{1}{2}\right) + (2g + 1)(-1)^{i_n}\left(\frac{1}{2}\right) = i_n(-1)^{i_n} \not= 0
\end{equation*}

\noindent then we can isolate the term $D_{(\vec{i}), 2g + 2}$, to get

\begin{align}
D_{(\vec{i}), 2g + 2} & = \frac{1}{i_n(-1)^{i_n + 1}}\left[ (-1)^{i_n + 1}\sum_{r = 1}^{i_n}\left[{g - i_n + r \choose r + 1} + \frac{1}{2}{g - i_n + r \choose r}\right]D_{(\vec{m}, i_n - r), 2g + 2} \right. \label{PHrecursion} \\
& \left. + \ \text{Cont}(\Gamma_1)\vert_{t = 1} + \sum_{\substack{2 \leq j \leq 2g - 2 \\ j \ \text{even}}}\text{Cont}(\Gamma_j)\vert_{t = 1} + \sum_{\substack{3 \leq j \leq 2g - 3 \\ j \ \text{odd}}} \text{Cont}(\Gamma_j)\vert_{t = 1} \right] \nonumber
\end{align}

\noindent The Theorem now follows from the following Lemma:

\begin{Lemma}\label{DegreesLessThan}
All Hodge integrals on the right hand side of Equation \ref{PHrecursion} involve Hodge monomials $\lambda_{\vec{v}}$ such that $|\vec{v}| < 2g - 1$.
\end{Lemma}

\begin{proof}
This is simply a consequence of meticulously analyzing all of the terms on the right hand side of Equation \ref{PHrecursion}, which we follow through with below. 

The first summation of Hodge integrals on the right hand side of Equation \ref{PHrecursion} involves integrals of the form $D_{(\vec{m}, i_n - r), 2g + 2}$, where $1 \leq r  \leq i_n$. Therefore, $|(\vec{m}, i_n - r)| = 2g - 1 - r < 2g - 1$. The same argument holds for all Hodge integrals appearing in $\text{Cont}(\Gamma_1)\vert_{t = 1}$. \\

\noindent Now consider the Hodge integrals appearing in 

\begin{equation*}
\sum_{\substack{2 \leq j \leq 2g - 2 \\ j \ \text{even}}} \text{Cont}(\Gamma_j)\vert_{t = 1}
\end{equation*}

\noindent Every term appearing in this sum is a product of Hodge integrals $D_{(\vec{v}_1), 2g_1 + 2}D_{(\vec{v}_2), 2g_2 + 2}$ where $g_1 + g_2 = g$ and $g_i > 0$. In particular, this means $g_i \leq g - 1$ Therefore, the Hodge integrals that occur are over moduli spaces of the form $\mbar_{0, (2g_i + 2)t}(\bztwo)$, but since $\text{dim}(\mbar_{0, (2g_i + 2)t}(\bztwo) = 2g_i - 1 \leq 2(g - 1) - 1 = 2g - 3$. this implies that the products of Hodge integrals appearing in this sum are non-zero if and only if the vectors $\vec{v}_i$ have the property that $|\vec{v}_i| \leq 2g - 3 < 2g - 1$, as desired. A completely analogous argument works for the Hodge integrals appearing in $\displaystyle \sum_{\substack{3 \leq j \leq 2g - 3 \\ j \ \text{odd}}}\text{Cont}(\Gamma_j)\vert_{t = 1}$. 

\end{proof}

\noindent With Lemma \ref{DegreesLessThan} established, this concludes the proof.

\end{proof}


\begin{proof}[Proof of Theorem \ref{Integrality}]

This proof requires simultaneous induction on $|\vec{i}|$ and $g$. However, since the recursions in Theorem \ref{NPHrecursion} are a two-step recursion, and Equation \ref{PHrecursion} involves many terms, our argument has many moving parts. To remedy this, let us summarize the argument:

\begin{enumerate}
\item{We induct on $|\vec{i}|$}
\item{There are two cases to consider, either $|\vec{i}|$ is odd, or $|\vec{i}|$ is even.}
\item{In the case that $|\vec{i}|$ is odd, this means $D_{(\vec{i}), 2g_0 + 2}$ is a pure Hodge integrals for some $g_0$. Using Equation \ref{PHrecursion}, we confirm the result holds in this base case, and then induct on $g$. Then we use the recursion obtained in Theorem \ref{NPHrecursion} to confirm that the result holds for the base case $d_{(\vec{i}), 2g_0 + 2}$, and then induct on $g$} 
\item{In the case that $|\vec{i}|$ is even, this means $D_{(\vec{i}), 2g + 2}$ is not a pure Hodge integral for any $g$, so we can solely refer to the recursions in Theorem \ref{NPHrecursion}. As in Step 3, we verify the result for base cases i.e. the smallest genus $g_0$ for which $D_{(\vec{i}), 2g_0+ 2}$ and $d_{(\vec{i}), 2g_0 + 2}$ are non-zero, and then induct on $g$}
\end{enumerate} 

\noindent Fix $\vec{i}$, and suppose that the integrality result holds for all vectors $\vec{v}$ where $|\vec{v}| < |\vec{i}|$.  \\

\noindent {\bf{Case 1}}: $|\vec{i}|$ is odd. \\

\noindent Since $|\vec{i}| = i_1 + i_2 + \ldots i_n$ is odd, this means there exists a $g_0$ such that $D_{(\vec{i}), 2g_0 + 2}$ is a pure Hodge integral. By definition, $g_0$ is the smallest genus $g$ for which $D_{(\vec{i}), 2g + 2}$ is non-zero. Similarly, $g_0$ is the smallest genus $g$ for which $d_{(\vec{i}), 2g + 2}$ is non-zero. Using Equation \ref{PHrecursion}, and multiplying through by $2^{|\vec{i}| + 1}i_n(-1)^{i_n + 1}$, we have  

\begin{align}
2^{|\vec{i}| + 1}i_n(-1)^{i_n + 1}D_{(\vec{i}), 2g_0 + 2} & = (-1)^{i_n + 1}2^{|\vec{i}| + 1}\sum_{r = 1}^{i_n}\left[{g_0 - i_n + r \choose r + 1} + \frac{1}{2}{g_0 - i_n + r \choose r}\right]D_{(\vec{m}, i_n - r), 2g_0 + 2} \label{normalized} \\
& + 2^{|\vec{i}| + 1}\text{Cont}(\Gamma_1)\vert_{t = 1} + 2^{|\vec{i}| + 1}\sum_{\substack{2 \leq j \leq 2g_0 - 2 \\ j \ \text{even}}}\text{Cont}(\Gamma_j)\vert_{t = 1} \nonumber \\
& + 2^{|\vec{i}| + 1}\sum_{\substack{3 \leq j \leq 2g_0 - 3 \\ j \ \text{odd}}}\text{Cont}(\Gamma_j)\vert_{t = 1} \nonumber
\end{align}

\noindent Consider the first term on the right hand side of Equation \ref{normalized}. It simplifies to 

\begin{align*}
(-1)^{i_n + 1} & \left[ 2^{|\vec{i}| + 1}\sum_{r = 1}^{i_n}{g_0 - i_n + r \choose r + 1}D_{(\vec{m}, i_n - r), 2g_0 + 2} + 2^{|\vec{i}|}\sum_{r = 1}^{i_n}{g_0 - i_n + r \choose r}D_{(\vec{m}, i_n - r), 2g_0 + 2} \right] \\
& = (-1)^{i_n + 1}\left[\sum_{r = 1}^{i_n}{g_0 - i_n + r \choose r + 1}2^r2^{|\vec{m}| + i_n - r + 1}D_{(\vec{m}, i_n - r), 2g_0 + 2} \right. \\
& \left. + \sum_{r = 1}^{i_n}{g_0 - i_n + r \choose r} 2^{r-1}2^{|\vec{m}| + i_n - r + 1}D_{(\vec{m}, i_n - r), 2g_0 + 2} \right]
\end{align*}

\noindent which is a sum of integers by Lemma \ref{DegreesLessThan} and the induction hypothesis, and therefore, an integer.  

Now consider the term $\displaystyle 2^{|\vec{i}| + 1}\text{Cont}(\Gamma_1)\vert_{t = 1}$. We have

\begin{align*}
2^{|\vec{i}| + 1}\text{Cont}(\Gamma_1)\vert_{t = 1} = (2g_0 + 1)(-1)^{i_n}\left[\sum_{r = 1}^{i_n}{g_0 - i_n + r \choose r + 1}2^r2^{|\vec{m}| + i_n - r + 1}d_{(\vec{m}, i_n - r), 2(g_0 - 1) + 2}\right]
\end{align*}

\noindent which is a sum of integers by Lemma \ref{DegreesLessThan} and the induction hypothesis, and therefore, an integer. \\

\noindent The term $\displaystyle 2^{|\vec{i}| + 1} \sum_{\substack{2 \leq j \leq 2g_0 - 2 \\ j \ \text{even}}}\text{Cont}(\Gamma_j)\vert_{t = 1}$ is equal to

\begin{align}
& (-1)^{i_n}\sum_{\substack{g_1 + g_2 = g, \ g_i > 0 \\ \vec{\ell}_1 + \vec{\ell}_2 = \vec{m} \\ p + q = i_n + 1 \\ 0 \leq r \leq p}} (-1)^{2g_2 - 1- q - |\vec{\ell}_2|}{2g + 1 \choose 2g_2}{g_1 - p + r \choose r} \label{secondterm} \\
& \times 2^{|\vec{i}| - |\vec{\ell}_1| - p + r - |\vec{\ell}_2| - q}2^{|\vec{\ell}_1| + p - r + 1}D_{(\vec{\ell}_1, p - r), 2g_1 + 2}2^{|\vec{\ell}_2| + q + 1}D_{(\vec{\ell}_2, q), 2g_2 + 2}  \nonumber \\
& + (-1)^{i_n}\sum_{\substack{g_1 + g_2 = g, \ g_i > 0 \\ \vec{\ell}_1 + \vec{\ell}_2 = \vec{m} \\ p + q = i_n \\ 0 \leq r \leq p}} (-1)^{2g_2 - 1- q - |\vec{\ell}_2|}{2g + 1 \choose 2g_2}{g_1 - p + r \choose r} \nonumber \\
& \times 2^{|\vec{i}| - |\vec{\ell}_1| - p + r - |\vec{\ell}_2| - q - 1}2^{|\vec{\ell}_1| + p - r + 1}D_{(\vec{\ell}_1, p - r), 2g_1 + 2}2^{|\vec{\ell}_2| + q + 1}D_{(\vec{\ell}_2, q), 2g_2 + 2} \nonumber
\end{align}

\noindent The integrality of the first summation in Equation \ref{secondterm} follows if $|\vec{i}| - |\vec{\ell}_1| - p + r - |\vec{\ell}_2| - q \geq 0$. Indeed,

\begin{align*}
|\vec{i}| - |\vec{\ell}_1| - p + r - |\vec{\ell}_2| - q  & = |\vec{m}| + i_n - (|\vec{\ell}_1| + |\vec{\ell}_2|) - (p + q) + r \\
& = |\vec{m}| - |\vec{m}| + i_n - (i_n + 1) + r  \\
& = r - 1
\end{align*}

\noindent By Lemma \ref{DegreesLessThan}, we have

\begin{align*}
& (|\vec{\ell}_1| + p - r) + (|\vec{\ell}_2| + q) < 2|\vec{i}| \\
\implies & |\vec{m}| + (p + q) - r < 2|\vec{i}| \\
\implies & |\vec{m}| + i_n + 1 - r < 2(|\vec{m}| + i_n) = 2|\vec{m}| + 2i_n\\
\implies & 1 - r < |\vec{m}| + i_n \\
\implies & r - 1 > |\vec{i}| \geq 0
\end{align*}

\noindent as desired. The integrality of the second summation in Equation \ref{secondterm} follows if $|\vec{i}| - |\vec{\ell}_1| - p + r - |\vec{\ell}_2| - q -1 \geq 0$. Indeed,

\begin{align*}
|\vec{i}| - |\vec{\ell}_1| - p + r - |\vec{\ell}_2| - q -1 & = |\vec{m}| + i_n - (|\vec{\ell}_1| + |\vec{\ell}_2|) - (p + q) + (r - 1) \\
& = |\vec{m}| - |\vec{m}| + i_n - i_n + r - 1 \\
& = r - 1 \\
& \geq 0
\end{align*}

\noindent as desired. 

When confirming the integrality of the term 

\begin{equation*}
2^{|\vec{i}| + 1}\sum_{\substack{3 \leq j \leq 2g_0 - 3 \\ j \ \text{odd}}}\text{Cont}(\Gamma_j)\vert_{t = 1}
\end{equation*}

\noindent one uses similar/analagous calculations as the ones used in the analysis of Equation \ref{secondterm}.


\noindent Therefore, we conclude that $2^{|\vec{i}| + 1}D_{(\vec{i}), 2g_0 + 2}$ is an integer. The remaining base case is the integrality of $d_{(\vec{i}), 2g_0 + 2}$.  Using Theorem \ref{NPHrecursion}, specializing to the parameter $k = 0$, and multiplying through by $2^{|\vec{i}| + 1}$, we obtain 

\begin{align}
2^{|\vec{i}| + 1}d_{(\vec{i}), 2g_0 + 2} & = \sum_{\substack{g_1 + g_2 = g_0 \\ 0 \leq g_2 \leq g - 1 \\ \vec{\ell} \leq \vec{i}}}(-1)^{|\vec{\ell}|}{2g_0 \choose 2g_2 + 1}2^{|\vec{i}| - |\vec{\ell}|  + 1}D_{(\vec{i} - \vec{\ell}), 2g_1 + 2}(2)^{|\vec{\ell}| + 1}D_{(\vec{\ell}), 2g_2 + 2} \label{basecaseuntwisted} \\
& + \sum_{\substack{g_1 + g_2 = g_0 - 1 \\ 0 \leq g_2 \leq g - 1 \\ \vec{\ell} \leq \vec{i}}}(-1)^{|\vec{\ell}|}{2g_0 \choose 2g_2 + 2}(2)^{|\vec{i}| - |\vec{\ell}| + 1}(2)^{|\vec{\ell}| + 1}d_{(\vec{i} - \vec{\ell}), 2g_1 + 2}d_{(\vec{\ell}), 2g_2 + 2} \nonumber
\end{align}

The summations on the right hand side of Equation \ref{basecaseuntwisted} are all integers by the induction hypothesis, except for the cases when $\vec{\ell} = \vec{0}$ and $\vec{\ell} = \vec{i}$. However, the integrals corresponding to these terms are zero for dimension reasons. Therefore, we conclude that $d_{(\vec{i}), 2g_0 + 2}$ is an integer. 

With the base cases established, we now induct on $g$. Suppose that the integrality results holds for $D_{(\vec{i}), 2g + 2}$ and $d_{(\vec{i}), 2g + 2}$ for all $g < \widetilde{g}$, where $\widetilde{g} > g_0$. Notice that $D_{(\vec{i}), 2\widetilde{g} + 2}$ is not a pure-Hodge integral, which means we can use Theorem \ref{NPHrecursion} to calculate it. Specializing to the parameter $k = 0$ in Theorem \ref{NPHrecursion}, we have

\begin{align}
2^{|\vec{i}| + 1}D_{\vec{i}, 2\widetilde{g} + 2} & = \sum_{\substack{g_1 + g_2 = \widetilde{g} - 1 \\ 0 \leq g_2 \leq \widetilde{g} - 1 \\ \vec{\ell} \leq \vec{i}}}(-1)^{|\vec{\ell}|}{2\widetilde{g} - 1 \choose 2g_2 + 1}(2)^{|\vec{i}| - |\vec{\ell}| + 1}d_{(\vec{i} - \vec{\ell}), 2g_1 + 2}(2)^{|\vec{\ell}| + 1}d_{(\vec{\ell}), 2g_2 + 2} \label{inductong} \\
& - \sum_{\substack{g_1 + g_2 = \widetilde{g} \\ 1 \leq g_2 \leq \widetilde{g} - 1 \\ \vec{\ell} \leq \vec{i}}}(-1)^{|\vec{\ell}|}{2\widetilde{g} - 1 \choose 2g_2}(2)^{|\vec{i}| - |\vec{\ell}| + 1}D_{(\vec{i} - \vec{\ell}), 2g_1 + 2}(2)^{|\vec{\ell}| + 1}D_{\vec{\ell}, 2g_2 + 2} \nonumber
\end{align}

\begin{align}
2^{|\vec{i}| + 1}d_{\vec{i}, 2\widetilde{g} + 2} & = \sum_{\substack{g_1 + g_2 = \widetilde{g} \\ 0 \leq g_2 \leq \widetilde{g} - 1 \\ \vec{\ell} \leq \vec{i}}} (-1)^{|\vec{\ell}|}{2\widetilde{g} \choose 2g_2 + 1}(2)^{|\vec{i}| - |\vec{\ell}| + 1}D_{(\vec{i} - \vec{\ell}), 2g_1 + 2}(2)^{|\vec{\ell}| + 1}D_{\vec{\ell}, 2g_2 + 2} \label{inductong2}\\
& \sum_{\substack{g_1 + g_2 = \widetilde{g} - 1 \\ 0 \leq g_2 \leq \widetilde{g} - 1 \\ \vec{\ell} \leq \vec{i}}} (-1)^{|\vec{\ell}|}{2\widetilde{g} \choose 2g_2 + 2}(2)^{|\vec{i}| - |\vec{\ell}| + 1}d_{(\vec{i} - \vec{\ell}), 2g_1 + 2}(2)^{|\vec{\ell}| + 1}d_{\vec{\ell}, 2g_2 + 2} \nonumber
\end{align}

\noindent The right hand sides of both equations above are all integers, either by the induction hypothesis on $|\vec{i}|$, or the induction hypothesis on $g < \widetilde{g}$. \\

\noindent This concludes the case when $|\vec{i}|$ is odd. \\

\noindent {\bf{Case 2}}: $|\vec{i}|$ is even. \\

\noindent As before, define $g_0$ as the smallest genus $g$ for which $D_{(\vec{i}), 2g + 2}$ and $d_{(\vec{i}), 2g + 2}$ are non-zero. Since $|\vec{i}|$ is even, this means $D_{(\vec{i}), 2g + 2}$ is not a pure Hodge integral for all $g$. Therefore, we can solely refer to Theorem \ref{NPHrecursion} for its computation. 

We again specialize to the parameter $k = 0$ in Theorem \ref{NPHrecursion}. To compute $(2)^{|\vec{i}| + 1}D_{(\vec{i}), 2g_0 + 2}$, we have

\begin{align*}
2^{|\vec{i}| + 1}D_{\vec{i}, 2g_0 + 2} & = \sum_{\substack{g_1 + g_2 = g_0 - 1 \\ 0 \leq g_2 \leq g_0 - 1 \\ \vec{\ell} \leq \vec{i}}}(-1)^{|\vec{\ell}|}{2g_0 - 1 \choose 2g_2 + 1}(2)^{|\vec{i}| - |\vec{\ell}| + 1}d_{(\vec{i} - \vec{\ell}), 2g_1 + 2}(2)^{|\vec{\ell}| + 1}d_{(\vec{\ell}), 2g_2 + 2} \\
& - \sum_{\substack{g_1 + g_2 = g_0 \\ 1 \leq g_2 \leq g_0 - 1 \\ \vec{\ell} \leq \vec{i}}}(-1)^{|\vec{\ell}|}{2g_0 - 1 \choose 2g_2}(2)^{|\vec{i}| - |\vec{\ell}| + 1}D_{(\vec{i} - \vec{\ell}), 2g_1 + 2}(2)^{|\vec{\ell}| + 1}D_{\vec{\ell}, 2g_2 + 2}
\end{align*}

\noindent The sums above involve summation over vectors $\vec{\ell}$ such that $\vec{\ell} \leq \vec{i}$. However, we can refine the inequality, and turn the above into sums over $\vec{\ell}$ such that $\vec{0} < \vec{\ell} < \vec{i}$. This follows for dimension reasons: when $\vec{\ell} = \vec{0}$ or $\vec{\ell} = \vec{i}$, the resulting terms vanish for dimension reasons. Since the summation is refined to $\vec{0} < \vec{\ell} < \vec{i}$, the integrality result holds for $D_{\vec{i}, 2g_0 + 2}$ by the induction hypothesis on $|\vec{i}|$. An analogous argument goes through when we compute $2^{|\vec{i}| + 1}d_{(\vec{i}), 2g_0 + 2}$. This establishes the base cases. When we induct on $g$, the same argument/computations used in the analysis of Equation \ref{inductong} and Equation \ref{inductong2} works.  \\

\noindent This concludes the proof of Theorem \ref{Integrality}

\end{proof}

\noindent Before we begin the proof of Theorem \ref{Polynomiality}, we establish a few useful lemmas along the way. Using the theory of integer valued polynomials (\cite{IVP}), we have the following lemma:

\begin{Lemma}\label{integervalued}
Suppose $a_g$ is an integer valued polynomial in $g$ of degree at most $i$. Then there exists integers $c_0, \ldots c_i$ such that

\begin{equation*}
a_g = \sum_{k = 0}^ic_k{g \choose k}
\end{equation*}

\flushright{$\square$}

\end{Lemma}

\begin{Lemma}\label{binexpfunc}
The exponential generating function of the binomial coefficient ${g \choose k}$ is 

\begin{equation*}
f(t) := \sum_{g \geq k}{g \choose k}\frac{t^g}{g!} = \frac{t^k}{k!}e^t
\end{equation*}

\end{Lemma}

\begin{proof}
This follows from direct computation:

\begin{align*}
\frac{t^k}{k!}e^t & = \frac{t^k}{k!}\sum_{g\geq 0}\frac{t^g}{g!} \\
& = \sum_{g \geq 0}\frac{t^{g + k}}{k!g!} \\
& = \sum_{g \geq k}\frac{t^g}{k!(g - k)!} \\
& = \sum_{g \geq k}\frac{g!}{g!}\frac{t^g}{k!(g - k)!} \\
& = \sum_{g \geq k} {g \choose k}\frac{t^g}{g!}
\end{align*}
\end{proof}

\begin{Lemma}\label{CombinatorialIdentity}
Let $0 \leq k \leq n$. We have the following combinatorial identity,

\begin{equation*}
\sum_{\ell + m = 2(n - k)}(-1)^\ell{n \choose \ell}{n \choose m} = (-1)^k{n \choose k}
\end{equation*}

\end{Lemma}

\begin{proof}
We use the notation $[x^i]p(x)$ to mean the degree $i$ coefficient of the polynomial $p(x)$. The Lemma follows from direct computation, as shown below,

\begin{align*}
\sum_{\ell + m = 2(n - k)}(-1)^\ell{n \choose \ell}{n \choose m} & = [x^{2(n-k)}](x - 1)^n (x + 1)^n \\
& = [x^{2(n - k)}](x^2 - 1)^n \\
& = [x^{2(n - k)}]\sum_{\ell = 0}^n(-1)^{n - \ell}{n \choose \ell}(x^2)^\ell \\
& = [x^{2(n - k)}]\sum_{\ell = 0}^n(-1)^{n - \ell}{n \choose \ell}(x^{2\ell}) \\
& = (-1)^{n - (n - k)}{n \choose n - k} \\
& = (-1)^k{n \choose k}
\end{align*}
\end{proof}

\begin{Lemma}\label{PolynomialityExponential}
Let $a_g$ be an integer valued sequence, and let $f(t) = \sum_{g \geq 0}a_g\frac{t^g}{g!}$ be its exponential generating function. Then $a_g$ is a polynomial in $g$ of degree at most $i$ if and only if there exists a polynomial $p(t)$ of degree at most $i$ such that $f(t) = p(t)e^t$.
\end{Lemma}

\begin{proof}

\noindent We first prove the forward direction. Suppose that $a_g$ is a polynomial in $g$ of degree at most $i$. By Lemma \ref{integervalued}, there exists constant $c_0, \ldots c_i$ such that

\begin{equation*}
a_g = \sum_{k = 0}^ic_k{g \choose k}
\end{equation*}

\noindent and therefore, the exponential generating function of $a_g$ is 

\begin{align*}
f(t) & := \sum_{g \geq 0}a_g\frac{t^g}{g!} \\
& = \sum_{g \geq 0}\left(\sum_{k = 0}^ic_k{g \choose k}\right)\frac{t^g}{g!} \\
& = \sum_{k = 0}^ic_k\left(\sum_{g \geq 0}{g \choose k}\frac{t^g}{g!}\right) \\
(\text{by Lemma \ref{binexpfunc}}) \hspace{1cm} & = \sum_{k = 0}^ic_k\frac{t^k}{k!}e^t \\
& = \left( \sum_{k = 0}^ic_k\frac{t^k}{k!} \right)e^t
\end{align*}

\noindent Since $\displaystyle \sum_{k = 0}^ic_k\frac{t^k}{k!}$ is a polynomial of degree at most $i$, this proves the forward direction. \\

Now we prove the converse. Suppose there exists a polynomial $p(t)$ of degree at most $i$ such that $\displaystyle f(t) := \sum_{g \geq 0}a_g\frac{t^g}{g!} = p(t)e^t$. If we express $p(t)$ as $p(t) = \displaystyle \sum_{n = 0}^ib_nt^n$, then

\begin{align*}
f(t) & = p(t)e^t = \left(\sum_{n = 0}^ib_nt^n\right)\left(\sum_{g \geq 0}\frac{t^g}{g!}\right) = \sum_{n = 0}^i\left( \sum_{g \geq 0}b_n\frac{t^{g + n}}{g!}\right) \\
& = \sum_{n = 0}^i\left(\sum_{g \geq 0}b_n(g + n)(g + (n - 1))\ldots(g + 1)\frac{t^{g + n}}{(g + n)!}\right) \\
& = \sum_{n = 0}^i\left( \sum_{g \geq n}b_n (g)(g - 1)\ldots(g - (n - 1))\frac{t^g}{g!} \right) = \sum_{n = 0}^i\left(\sum_{g \geq n}b_nn!{g \choose n}\frac{t^g}{g!}\right) \\
\left(\text{since} \ {g \choose n} = 0 \ \text{for} \ g < n\right)  \ & = \sum_{n = 0}^i\left(\sum_{g \geq 0}b_nn!{g \choose n}\frac{t^g}{g!}\right) = \sum_{g \geq 0}\left(\sum_{n = 0}^ib_nn!{g \choose n}\right)\frac{t^g}{g!}
\end{align*}

\noindent Since $\displaystyle \sum_{n = 0}^ib_nn!{g \choose n}$ is a polynomial in $g$ of degree at most $i$, the Lemma follows.
\end{proof}


\begin{proof}[Proof of Theorem \ref{Polynomiality}]

Throughout, we define 

\begin{align*}
& F_{\vec{i}}(t) := \sum_{g \geq 0}D_{(\vec{i}), 2g + 2}\frac{t^g}{g!} \\
& G_{\vec{i}}(t) := \sum_{g \geq 0}d_{(\vec{i}), 2g + 2}\frac{t^g}{g!}
\end{align*}

\noindent The proof proceeds as follows:

\begin{enumerate}
\item{We consider the recursions obtained in Theorem \ref{NPHrecursion}. For the recursion in which $D_{(\vec{i}), 2g + 2}$ is the principal part, we specialize to the parameter $k = 2g - 2 - |\vec{i}|$, and for the recursion in which $d_{(\vec{i}), 2g + 2}$ is the principal part, we specialize to the parameter $k = 2g - 1 - |\vec{i}|$.}
\item{We translate the recursions into systems of ordinary differential equations for $F_{\vec{i}}$ and $G_{\vec{i}}$}
\item{We use induction on $|\vec{i}|$, and in the Laplace space, we show that $\mathcal{L}\left\{F_{\vec{i}}\right\}$ and $\mathcal{L}\left\{G_{\vec{i}}\right\}$ satisfy the desired result.}
\end{enumerate}

When specializing to the parameter $k = 2g - 2 - |\vec{i}|$ in Equation \ref{NPHtwisted}, and to the parameter $k = 2g - 1 - |\vec{i}|$ in Equation \ref{NPHuntwisted}, we have

\begin{align*}
D_{(\vec{i}), 2g + 2} & = \sum_{g_2 = 0}^{\left\lfloor \frac{|\vec{i}|}{2} \right\rfloor}{|\vec{i}| + 1 \choose 2g_2 + 1}d_{(\vec{i}), 2(g - (g_2 + 1)) + 2} - \sum_{g_2 = 1}^{\left\lfloor \frac{|\vec{i}| + 1}{2} \right\rfloor} {|\vec{i}| + 1 \choose 2g_2}D_{(\vec{i}), 2(g - g_2) + 2} \\
& + 2\sum_{g_2 = 0}^{\left\lfloor \frac{|\vec{i}|}{2} \right\rfloor}\sum_{0 < \vec{\ell} \leq \vec{i}}(-1)^{|\vec{\ell}|}{|\vec{i}| + 1 \choose 2g_2 + 1}d_{(\vec{i} - \vec{\ell}), 2(g - (g_2 + 1)) + 2}d_{(\vec{\ell}), 2g_2 + 2} \\
& - 2\sum_{g_2 = 1}^{\left\lfloor \frac{|\vec{i}| + 1}{2} \right\rfloor} \sum_{0 < \vec{\ell} \leq \vec{i}} (-1)^{|\vec{\ell}|}{|\vec{i}| + 1 \choose 2g_2}D_{(\vec{i} - \vec{\ell}), 2(g - g_2) + 2}D_{(\vec{\ell}), 2g_2 + 2}
\end{align*}

\begin{align*}
d_{(\vec{i}), 2g + 2} & = \sum_{g_2 = 0}^{\left\lfloor \frac{|\vec{i}|}{2} \right\rfloor}{|\vec{i}| + 1 \choose 2g_2 + 1}D_{(\vec{i}), 2(g - g_2) + 2} - \sum_{g_2 = 0}^{\left\lfloor \frac{|\vec{i}| + 1}{2} \right\rfloor - 1}{|\vec{i}| + 1 \choose 2g_2 + 2}d_{(\vec{i}), 2(g - (g_2 + 1))} \\
& + 2\sum_{g_2 = 0}^{\left\lfloor \frac{|\vec{i}|}{2} \right\rfloor}\sum_{0 < \vec{\ell} \leq i}(-1)^{|\vec{\ell}|}{|\vec{i}| + 1 \choose 2g_2 + 1}D_{(\vec{i} - \vec{\ell}), 2(g - g_2) + 2}D_{(\vec{\ell}), 2g_2 + 2} \\
& - 2 \sum_{g_2 = 0}^{\left\lfloor \frac{|\vec{i}| + 1}{2} \right\rfloor - 1}\sum_{0 < \vec{\ell} \leq \vec{i}}(-1)^{|\vec{\ell}|}{|\vec{i}| + 1 \choose 2g_2 + 2}d_{(\vec{i} - \vec{\ell}), 2(g - (g_2 + 1))}d_{(\vec{\ell}), 2g_2 + 2}
\end{align*}

\noindent Suppose that the polynomiality result holds for all vectors $\vec{v}$ such that $|\vec{v}| < |\vec{i}|$, for some vector $\vec{i}$. Translating the above recursions into a system of ordinary differential equations for $F_{\vec{i}}$ and $G_{\vec{i}}$, and applying Lemma \ref{PolynomialityExponential}, there exist polynomials $P_{\vec{i}}(t)$ and $Q_{\vec{i}}(t)$ of degree at most $(|\vec{i}| - 1)^2 + 1$ such that

\begin{equation*}
\partial_t^{\left\lfloor \frac{|\vec{i}|}{2} \right\rfloor + 1}F_{\vec{i}} = \sum_{g_2 = 0}^{\lfloor \frac{|\vec{i}|}{2} \rfloor} {|\vec{i}| + 1 \choose 2g_2 + 1}\partial_t^{\lfloor \frac{|\vec{i}|}{2} \rfloor - g_2}G_{\vec{i}} - \sum_{g_2 = 1}^{\lfloor \frac{|\vec{i}| + 1}{2} \rfloor}{|\vec{i}| + 1 \choose 2g_2}\partial_t^{\left\lfloor \frac{|\vec{i}|}{2}\right\rfloor + 1 - g_2}F_{\vec{i}} + P_{\vec{i}}(t)e^t
\end{equation*}

\begin{equation*}
\partial_t^{\left\lfloor \frac{|\vec{i}| + 1}{2} \right\rfloor}G_{\vec{i}} = \sum_{g_2 = 0}^{\left\lfloor \frac{|\vec{i}|}{2} \right\rfloor}{|\vec{i}| + 1 \choose 2g_2 + 1}\partial_t^{\left\lfloor \frac{|\vec{i}| + 1}{2} \right\rfloor - g_2}F_{\vec{i}} - \sum_{g_2 = 0}^{\left\lfloor \frac{|\vec{i}| + 1}{2} \right\rfloor - 1}{|\vec{i}| + 1 \choose 2g_2 + 2}\partial_t^{\left\lfloor \frac{|\vec{i}| + 1}{2} \right\rfloor - (g_2 + 1)}G_{\vec{i}} + Q_{\vec{i}}(t)e^t
\end{equation*}

\noindent Denote by $\mathcal{L}$ the Laplace transform, and define $\widetilde{F}_{\vec{i}}(s) := \mathcal{L}\{F_{\vec{i}}(t)\}(s)$, and $\widetilde{G}_{\vec{i}}(s) := \mathcal{L}\{G_{\vec{i}}(t)\}(s)$. Recall that $\mathcal{L}\{ t^ne^t\}(s) = \frac{n!}{(s - 1)^{n + 1}}$. Using this fact, when we take the Laplace transform of the above and combine like terms, we see that there exist constants $a_k, b_k$, for $0 \leq k \leq (|\vec{i}| - 1)^2 + 1$, such that 

\begin{equation*}
\left( \sum_{g_2 = 0}^{\left\lfloor \frac{|\vec{i}| + 1}{2} \right\rfloor} {|\vec{i}| + 1 \choose 2g_2} s^{\left\lfloor \frac{|\vec{i}|}{2} \right\rfloor + 1 - g_2} \right)\widetilde{F}_{\vec{i}}(s) = \left( \sum_{g_2 = 0}^{\left\lfloor \frac{|\vec{i}|}{2} \right\rfloor} {|\vec{i}| + 1 \choose 2g_2 + 1} s^{\left\lfloor \frac{|\vec{i}|}{2} \right\rfloor - g_2}\right)\widetilde{G}_{\vec{i}}(s) + \sum_{k = 0}^{(|\vec{i}| - 1)^2 + 1}\frac{a_k}{(s - 1)^{k + 1}}
\end{equation*}

\begin{equation*}
\left( \sum_{g_2 = 0}^{\left\lfloor \frac{|\vec{i}| + 1}{2} \right\rfloor} {|\vec{i}| + 1 \choose 2g_2} s^{\left\lfloor \frac{|\vec{i}| + 1}{2} \right\rfloor - g_2} \right)\widetilde{G}_{\vec{i}}(s) = \left( \sum_{g_2 = 0}^{\left\lfloor \frac{|\vec{i}|}{2} \right\rfloor} {|\vec{i}| + 1 \choose 2g_2 + 1} s^{\left\lfloor \frac{|\vec{i}| + 1}{2} \right\rfloor - g_2} \right)\widetilde{F}_{\vec{i}}(s) + \sum_{k = 0}^{(|\vec{i}| - 1)^2 + 1}\frac{b_k}{(s - 1)^{k + 1}}
\end{equation*}

\noindent To ease notation, we make the following definitions:

\begin{align*}
& A_{\vec{i}, 1} :=  \sum_{g_2 = 0}^{\left\lfloor \frac{|\vec{i}| + 1}{2} \right\rfloor} {|\vec{i}| + 1 \choose 2g_2} s^{\left\lfloor \frac{|\vec{i}|}{2} \right\rfloor + 1 - g_2} \\
& A_{\vec{i}, 2} := \sum_{g_2 = 0}^{\left\lfloor \frac{|\vec{i}|}{2} \right\rfloor} {|\vec{i}| + 1 \choose 2g_2 + 1} s^{\left\lfloor \frac{|\vec{i}|}{2} \right\rfloor - g_2} \\
& A_{\vec{i}, 3} := \sum_{g_2 = 0}^{\left\lfloor \frac{|\vec{i}| + 1}{2} \right\rfloor} {|\vec{i}| + 1 \choose 2g_2} s^{\left\lfloor \frac{|\vec{i}| + 1}{2} \right\rfloor - g_2} \\
& A_{\vec{i}, 4} := \sum_{g_2 = 0}^{\left\lfloor \frac{|\vec{i}|}{2} \right\rfloor} {|\vec{i}| + 1 \choose 2g_2 + 1} s^{\left\lfloor \frac{|\vec{i}| + 1}{2} \right\rfloor - g_2}
\end{align*}

\noindent After some simplification, we have

\begin{align*}
& (A_{\vec{i}, 1}A_{\vec{i}, 3} - A_{\vec{i}, 2}A_{\vec{i}, 4})\widetilde{F}_{\vec{i}}(s) = A_{\vec{i}, 2}\sum_{k = 0}^{(|\vec{i}| - 1)^2 + 1}\frac{b_k}{(s - 1)^{k + 1}} + A_{\vec{i}, 3}\sum_{k = 0}^{(|\vec{i}| - 1)^2 + 1}\frac{a_k}{(s - 1)^{k +1}} \\
& (A_{\vec{i}, 1}A_{\vec{i}, 3} - A_{\vec{i}, 2}A_{\vec{i}, 4})\widetilde{G}_{\vec{i}}(s) = A_{\vec{i}, 4}\sum_{k = 0}^{(|\vec{i}| - 1)^2 + 1}\frac{a_k}{(s - 1)^{k + 1}} + A_{\vec{i}, 1}\sum_{k = 0}^{(|\vec{i}| - 1)^2 + 1}\frac{b_k}{(s - 1)^{k + 1}}
\end{align*}

\noindent {\bf{Case 1}}: $|\vec{i}|$ \ \text{is odd} \\

\noindent In this case, $\left\lfloor \frac{|\vec{i}| + 1}{2} \right\rfloor = \frac{|\vec{i}| + 1}{2}$ and $\left\lfloor \frac{|\vec{i}|}{2} \right\rfloor = \frac{|\vec{i}| - 1}{2}$, so

\begin{align*}
A_{\vec{i}, 1}A_{\vec{i}, 3} - A_{\vec{i}, 2}A_{\vec{i}, 4} & = \left(\sum_{g_2 = 0}^{\frac{|\vec{i}| + 1}{2}}{|\vec{i}| + 1 \choose 2g_2}s^{\frac{|\vec{i}| + 1}{2} - g_2}\right)^2  - \left(\sum_{g_2 = 0}^{\frac{|\vec{i}| - 1}{2}} {|\vec{i}| + 1 \choose 2g_2 + 1}s^{\frac{|\vec{i}| - 1}{2} - g_2}\right) \left( \sum_{g_2 = 0}^{\frac{|\vec{i}| - 1}{2}} {|\vec{i}| + 1 \choose 2g_2 + 1} s^{\frac{|\vec{i}| + 1}{2} - g_2} \right) \\
& = \sum_{k = 0}^{|\vec{i}| + 1}\left( \sum_{\ell + m = |\vec{i}| + 1 - k} {|\vec{i}| + 1 \choose 2\ell}{|\vec{i}| + 1 \choose 2m} \right)s^k -  \sum_{k = 0}^{|\vec{i}|}\left( \sum_{\ell + m = |\vec{i}| - k} {|\vec{i}| + 1 \choose 2\ell + 1}{|\vec{i}| + 1 \choose 2m + 1} \right)s^k \\
& = \sum_{k = 0}^{|\vec{i}| + 1}\left( \sum_{\ell + m = |\vec{i}| + 1 - k}{|\vec{i}| + 1 \choose 2\ell}{|\vec{i}| + 1 \choose 2m} - \sum_{\ell + m = |\vec{i}| - k}{|\vec{i}| + 1 \choose 2\ell + 1}{|\vec{i}| + 1 \choose 2m + 1}  \right)s^k \\
& = \sum_{k = 0}^{|\vec{i}| + 1}\left( \sum_{\ell + m = 2(|\vec{i}| + 1 - k)} (-1)^{\ell}{|\vec{i}| + 1 \choose \ell}{|\vec{i}| + 1 \choose m} \right)s^k \\
(\text{by Lemma \ref{CombinatorialIdentity}}) & = \sum_{k = 0}^{|\vec{i}| + 1}(-1)^k{|\vec{i}| + 1 \choose k}s^k \\
& = (s - 1)^{|\vec{i}| + 1}
\end{align*}

\noindent {\bf{Case 2}}: $|\vec{i}|$ is even \\

\noindent In this case, $\left\lfloor \frac{|\vec{i}| + 1}{2} \right\rfloor = \left\lfloor \frac{|\vec{i}|}{2} \right\rfloor =  \frac{|\vec{i}|}{2}$, so

\begin{align*}
A_{\vec{i}, 1}A_{\vec{i}, 3} - A_{\vec{i}, 2}A_{\vec{i}, 4} & = \left( \sum_{g_2 = 0}^{\frac{|\vec{i}|}{2}}{|\vec{i} + 1 \choose 2g_2}s^{\frac{|\vec{i}|}{2} + 1 - g_2} \right)\left( \sum_{g_2 = 0}^{\frac{|\vec{i}|}{2}} {|\vec{i}| + 1 \choose 2g_2} s^{\frac{|\vec{i}|}{2} - g_2} \right) - \left( \sum_{g_2 = 0}^{\frac{|\vec{i}|}{2}} {|\vec{i}| + 1 \choose 2g_2 + 1} s^{\frac{|\vec{i}|}{2} - g_2}\right)^2 \\
& = \sum_{k = 0}^{|\vec{i}| + 1}\left( \sum_{\ell + m = |\vec{i}| + 1 - k}{|\vec{i}| + 1 \choose 2\ell}{|\vec{i}| + 1 \choose 2m} \right)s^k - \sum_{k = 0}^{|\vec{i}|}\left( \sum_{\ell + m = |\vec{i}| - k} {|\vec{i}| + 1 \choose 2\ell + 1}{|\vec{i}| + 1 \choose 2m + 1} \right)s^k \\
& = \sum_{k = 0}^{|\vec{i}| + 1}\left( \sum_{\ell + m = |\vec{i}| + 1 - k} {|\vec{i}| + 1 \choose 2\ell}{|\vec{i}| + 1 \choose 2m}  - \sum_{\ell + m = |\vec{i}| - k}{|\vec{i}| + 1 \choose 2\ell + 1}{|\vec{i}| + 1 \choose 2m + 1}\right)s^k \\
& = \sum_{k = 0}^{|\vec{i}| + 1}\left( \sum_{\ell + m = 2(|\vec{i}| + 1 - k)} (-1)^{\ell}{|\vec{i}| + 1 \choose \ell}{|\vec{i}| + 1 \choose m} \right)s^k \\
(\text{by Lemma \ref{CombinatorialIdentity}}) & = \sum_{k = 0}^{|\vec{i}| + 1}(-1)^k{|\vec{i}| + 1 \choose k}s^k \\
& = (s - 1)^{|\vec{i}| + 1}
\end{align*}

\noindent Therefore, we have

\begin{align*}
& (s - 1)^{|\vec{i}| + 1}\widetilde{F}_{\vec{i}}(s) = A_{\vec{i}, 2}\sum_{k = 0}^{(|\vec{i}| - 1)^2 + 1}\frac{b_k}{(s - 1)^{k + 1}} + A_{\vec{i}, 3}\sum_{k = 0}^{(|\vec{i}| - 1)^2 + 1}\frac{a_k}{(s - 1)^{k +1}} \\
& (s - 1)^{|\vec{i}| + 1}\widetilde{G}_{\vec{i}}(s) = A_{\vec{i}, 4}\sum_{k = 0}^{(|\vec{i}| - 1)^2 + 1}\frac{a_k}{(s - 1)^{k + 1}} + A_{\vec{i}, 1}\sum_{k = 0}^{(|\vec{i}| - 1)^2 + 1}\frac{b_k}{(s - 1)^{k + 1}}
\end{align*}

\noindent Dividing through by $(s - 1)^{|\vec{i}| + 1}$, we get

\begin{align}
& \widetilde{F}_{\vec{i}}(s) = A_{\vec{i}, 2}\sum_{k = 0}^{(|\vec{i}| - 1)^2 + 1}\frac{b_k}{(s - 1)^{|\vec{i}| + k + 2}} + A_{\vec{i}, 3}\sum_{k = 0}^{(|\vec{i}| - 1)^2 + 1}\frac{a_k}{(s - 1)^{|\vec{i}| + k +2}} \label{KeyStep1} \\
& \widetilde{G}_{\vec{i}}(s) = A_{\vec{i}, 4}\sum_{k = 0}^{(|\vec{i}| - 1)^2 + 1}\frac{a_k}{(s - 1)^{|\vec{i}| + k + 2}} + A_{\vec{i}, 1}\sum_{k = 0}^{(|\vec{i}| - 1)^2 + 1}\frac{b_k}{(s - 1)^{|\vec{i}| + k + 2}} \label{KeyStep2}
\end{align}

\noindent The highest power of $(s - 1)$ on the right hand sides of the above equations is $|\vec{i}| + 2 + (|\vec{i} - 1)^2 + 1 = |\vec{i}|^2 - |\vec{i}| + 4$. Collecting the terms under this common denominator, we get

\begin{align*}
& \widetilde{F}_{\vec{i}}(s) = \frac{\displaystyle \sum_{k = 0}^{(|\vec{i}| - 1)^2 + 1} (A_{\vec{i}, 2}b_k  + A_{\vec{i}, 3}a_k)(s - 1)^{(|\vec{i}| - 1)^2 + 1 - k}}{(s - 1)^{|\vec{i}|^2 - |\vec{i}| + 4}} \\
& \widetilde{G}_{\vec{i}}(s) = \frac{\displaystyle \sum_{k = 0}^{(|\vec{i}| - 1)^2 + 1}(A_{\vec{i}, 4}a_k + A_{\vec{i}, 1}b_k)(s - 1)^{(|\vec{i}| - 1)^2 + 1 - k}}{(s - 1)^{|\vec{i}|^2 - |\vec{i}| + 4}}
\end{align*}

\noindent The next step is to make sure that the degrees of the polynomials occurring in the numerator of $\widetilde{F}_{\vec{i}}$ and $\widetilde{G}_{\vec{i}}$ are strictly less than $|\vec{i}|^2 - |\vec{i}| + 4$. \\

\noindent Lets first consider the numerator of $\widetilde{F}_{\vec{i}}$. The degree of this polynomial, which we denote by $d_1$, is 

\begin{equation*}
d_1 := \text{max}\left\{ \left\lfloor \frac{|\vec{i}|}{2} \right\rfloor, \left\lfloor \frac{|\vec{i}| + 1}{2} \right\rfloor \right\} + (|\vec{i}| - 1)^2 + 1 = \left\lfloor \frac{|\vec{i}| + 1}{2} \right\rfloor + (|\vec{i}| - 1)^2 + 1
\end{equation*}

\noindent If $|\vec{i}|$ is odd, the degree is $\frac{|\vec{i}| + 1}{2} + (|\vec{i}| - 1)^2 + 1 = |\vec{i}|^2 - \frac{3}{2}|\vec{i}| + \frac{5}{2}$, and if $\vec{i}$ is even, the degree is $\frac{|\vec{i}|}{2} + (|\vec{i}| - 1)^2 + 1 = |\vec{i}|^2 - \frac{3}{2}|\vec{i}| + 2$. In either case, $d_1 < |\vec{i}|^2 - |\vec{i}| + 4$. \\

\noindent Now consider the numerator of $\widetilde{G}_{\vec{i}}$, The degree of this polynomial, which we denote by $d_2$, is

\begin{equation*}
d_2 := \text{max}\left\{ \left\lfloor \frac{|\vec{i}|}{2} \right\rfloor + 1, \left\lfloor \frac{|\vec{i}| + 1}{2} \right\rfloor  \right\} + (|\vec{i}| - 1)^2 + 1
\end{equation*}

\noindent When $|\vec{i}|$ is odd, $\text{max}\left\{ \left\lfloor \frac{|\vec{i}|}{2} \right\rfloor + 1, \left\lfloor \frac{|\vec{i}| + 1}{2} \right\rfloor  \right\} = \frac{|\vec{i}| + 1}{2}$, and $d_2 = |\vec{i}|^2 -\frac{3}{2}|\vec{i}| + \frac{5}{2}$. When $|\vec{i}|$ is even, $d_2 = |\vec{i}|^2 - \frac{3}{2}|\vec{i} + 3$. In either case, $d_2 < |\vec{i}|^2 - |\vec{i}| + 4$. \\

\noindent Expanding the numerator polynomials of $\widetilde{F}_{\vec{i}}$ and $\widetilde{G}_{\vec{i}}$ at $s = 1$, we see that there exist constants $f_k, g_k$ such that

\begin{align*}
& \widetilde{F}_{\vec{i}}(s) = \frac{\displaystyle \sum_{k = 0}^{d_1} f_k(s - 1)^k}{(s - 1)^{|\vec{i}|^2 - |\vec{i}| + 4}} = \sum_{k = 0}^{d_1}\frac{f_k}{(s - 1)^{|\vec{i}|^2 - |\vec{i}| + 4 - k}} \\
& \widetilde{G}_{\vec{i}}(s) = \frac{\displaystyle \sum_{k = 0}^{d_2} g_k(s - 1)^k}{(s - 1)^{|\vec{i}|^2 - |\vec{i}| + 4}} = \sum_{k = 0}^{d_2}\frac{g_k}{(s - 1)^{|\vec{i}|^2 - |\vec{i}| + 4 - k}}
\end{align*}

\noindent Taking the inverse Laplace transform, we see that

\begin{align*}
& F_{\vec{i}}(t) = \left(\sum_{k = 0}^{d_1} \frac{f_k}{(|\vec{i}|^2 - |\vec{i}| + 3 - k)!}t^{|\vec{i}|^2 - |\vec{i}| + 4 - k}\right)e^t \\
& G_{\vec{i}}(t) = \left(\sum_{k = 0}^{d_2} \frac{g_k}{(\vec{i}|^2 - |\vec{i}| + 3 - k)!}t^{|\vec{i}|^2 - |\vec{i}| + 4 - k}\right)e^t
\end{align*}

\noindent Since $\displaystyle \sum_{k = 0}^{d_1} \frac{f_k}{(|\vec{i}|^2 - |\vec{i}| + 3 - k)!}t^{|\vec{i}|^2 - |\vec{i}| + 4 - k}$ and $\displaystyle \sum_{k = 0}^{d_2} \frac{g_k}{(\vec{i}|^2 - |\vec{i}| + 3 - k)!}t^{|\vec{i}|^2 - |\vec{i}| + 4 - k}$ are polynomials of degree at most $|\vec{i}|^2 - |\vec{i}| + 4$, by Lemma \ref{PolynomialityExponential}, we see that $D_{(\vec{i}), 2g + 2}$ and $d_{(\vec{i}), 2g + 2}$ are polynomial in $g$, with degree at most $|\vec{i}|^2 - 3|\vec{i}| + 4 \leq |\vec{i}|^2 + 1$, as desired.

\end{proof}


\begin{proof}[Proof of Theorem \ref{PDE}]

The Theorem follows from direct computation and Theorem \ref{NPHrecursion}.  \\

\noindent First, we compute the formal expressions of the functions appearing in Equation \ref{thepdes1} and Equation \ref{thepdes2} :

\begin{align*}
& \partial_tF(-\vec{s}, t) = \sum_{\substack{g \geq 0 \\ \vec{i} \geq \vec{0}}} (-1)^{|\vec{i}|}D_{(\vec{i}), 2g + 2}s^{\vec{i}}\frac{t^{2g + 1}}{(2g + 1)!}\\
& \partial_t^2F(-\vec{s}, t) = \sum_{\substack{g \geq 0 \\ \vec{i} \geq \vec{0}}} (-1)^{|\vec{i}|}D_{(\vec{i}), 2g + 2}s^{\vec{i}}\frac{t^{2g}}{(2g)!} \\
& \partial_t^3F(\vec{s}, t) = \sum_{\substack{g \geq 1 \\ \vec{i} \geq \vec{0}}} D_{(\vec{i}), 2g + 2} s^{\vec{i}}\frac{t^{2g - 1}}{(2g - 1)!} \\
& G(-\vec{s}, t) = \sum_{\substack{g \geq 0 \\ \vec{i} \geq \vec{0}}} (-1)^{|\vec{i}|}d_{(\vec{i}), 2g + 2}s^{\vec{i}}\frac{t^{2g + 2}}{(2g + 2)!} \\
& \partial_tG(-\vec{s}, t) = \sum_{\substack{g \geq 0 \\ \vec{i} \geq \vec{0}}} (-1)^{|\vec{i}|}d_{(\vec{i}), 2g + 2}s^{\vec{i}}\frac{t^{2g + 1}}{(2g + 1)!} \\
& \partial_t^2G(\vec{s}, t) = \sum_{\substack{g \geq 0 \\ \vec{i} \geq \vec{0}}} d_{(\vec{i}), 2g + 2}s^{\vec{i}}\frac{t^{2g}}{(2g)!}
\end{align*}

\noindent The right hand side of Equation \ref{thepdes1} is 

\begin{align*}
2\partial_t^2G(\vec{s}, t)\partial_tG(-\vec{s}, t) & = 2\left( \sum_{\substack{g \geq 0 \\ \vec{i} \geq \vec{0}}} d_{(\vec{i}), 2g + 2}s^{\vec{i}}\frac{t^{2g}}{(2g)!} \right)\left( \sum_{\substack{g \geq 0 \\ \vec{i} \geq \vec{0}}} (-1)^{|\vec{i}|}d_{(\vec{i}), 2g + 2}s^{\vec{i}}\frac{t^{2g + 1}}{(2g + 1)!} \right) \\
& = 2\sum_{\substack{g \geq 1 \\ \vec{i} \geq \vec{0}}}\left(\sum_{\substack{\vec{\ell}_1 + \vec{\ell}_2 = \vec{i} \\ 2g_1 + 2g_2 + 1 = 2g - 1}} (-1)^{|\vec{\ell}_2|}\frac{1}{(2g_1)!}\frac{1}{(2g_2 + 1)!}d_{(\vec{\ell}_1), 2g_1 + 2}d_{(\vec{\ell}_2), 2g_2 + 1}s^{\vec{\ell}_1}s^{\vec{\ell}_2} \right)t^{2g - 1} \\
& = 2\sum_{\substack{g \geq 1 \\ \vec{i} \geq 0}}\left( \sum_{\substack{g_1 + g_2 = g - 1 \\ 0 \leq g_2 \leq g \\ \vec{\ell} \leq \vec{i}}}(-1)^{|\vec{\ell}|}{2g - 1 \choose 2g_2 + 1}d_{(\vec{i} - \vec{\ell}), 2g_1 + 2}d_{(\vec{\ell}), 2g_2 + 2}\right)s^{\vec{i}}\frac{t^{2g - 1}}{(2g - 1)!}
\end{align*}

\noindent The left hand side of Equation \ref{thepdes1} is

\begin{align*}
2\partial_t^3F(\vec{s}, t)\partial_t^2F(-\vec{s}, t) & = 2\left( \sum_{\substack{g \geq 1 \\ \vec{i} \geq \vec{0}}} D_{(\vec{i}), 2g + 2} s^{\vec{i}}\frac{t^{2g - 1}}{(2g - 1)!} \right) \left( \sum_{\substack{g \geq 0 \\ \vec{i} \geq \vec{0}}} (-1)^{|\vec{i}|}D_{(\vec{i}), 2g + 2}s^{\vec{i}}\frac{t^{2g}}{(2g)!} \right) \\
& = 2\sum_{\substack{g \geq 1 \\ \vec{i} \geq \vec{0}}}\left( \sum_{\substack{2g_1 - 1 + 2g_2 = 2g - 1 \\ \vec{\ell}_2 + \vec{\ell}_2 = \vec{i}}} (-1)^{|\vec{\ell}_2|}\frac{1}{(2g_1 - 1)!}\frac{1}{(2g_2)!} D_{(\vec{\ell}_1), 2g_1 + 2}D_{(\vec{\ell}_2), 2g_2 + 2}s^{\vec{\ell}_1}s^{\vec{\ell}_2} \right)t^{2g - 1} \\
& = 2\sum_{\substack{g \geq 1 \\ \vec{i} \geq \vec{0}}}\left( \sum_{\substack{g_1 + g_2 = g \\ 1 \leq g_1 \leq g \\ \vec{\ell} \leq \vec{i}}} (-1)^{|\vec{\ell}|} {2g - 1 \choose 2g_2} D_{(\vec{i} - \vec{\ell}), 2g_1 + 2}D_{(\vec{\ell}), 2g_2 + 2} \right) s^{\vec{i}} \frac{t^{2g - 1}}{(2g - 1)!} \\
& = 2\sum_{\substack{g \geq 1 \\ \vec{i} \geq \vec{0}}}\left( \sum_{\substack{g_1 + g_2 = g \\ 1 \leq g_2 \leq g - 1 \\ \vec{\ell} \leq \vec{i}}} (-1)^{|\vec{\ell}|} {2g - 1 \choose 2g_2} D_{(\vec{i} - \vec{\ell}), 2g_1 + 2}D_{(\vec{\ell}), 2g_2 + 2} + \frac{1}{2}D_{(\vec{i}), 2g + 2} \right) s^{\vec{i}} \frac{t^{2g - 1}}{(2g - 1)!}
\end{align*}

\noindent And therefore, Equation \ref{thepdes1} follows by Equation \ref{NPHtwisted} in Theorem \ref{NPHrecursion} with $k = 0$. \\

\noindent Now we proceed with Equation \ref{thepdes2}. The left hand side of Equation \ref{thepdes2} is

\begin{align*}
2\partial_t^2G(\vec{s}, t)G(-\vec{s}, t) & = 2 \left( \sum_{\substack{g \geq 0 \\ \vec{i} \geq \vec{0}}} d_{(\vec{i}), 2g + 2}s^{\vec{i}}\frac{t^{2g}}{(2g)!}\right) \left( \sum_{\substack{g \geq 0 \\ \vec{i} \geq \vec{0}}} (-1)^{|\vec{i}|} d_{(\vec{i}), 2g + 2}s^{\vec{i}}\frac{t^{2g + 2}}{(2g + 2)!} \right) \\
& = 2\sum_{\substack{g \geq 1 \\ \vec{i} \geq \vec{0}}}\left( \sum_{\substack{2g_1 + 2g_2 + 2 = 2g \\ \vec{\ell}_1 + \vec{\ell}_2 = \vec{i}}} (-1)^{|\vec{\ell}_2|}\frac{1}{(2g_1)!}\frac{1}{(2g_2 + 2)!} d_{(\vec{\ell}_1), 2g_1 + 2}d_{(\vec{\ell}_2), 2g_2 + 2} s^{\vec{\ell}_1}s^{\vec{\ell}_2} \right)t^{2g} \\
& = 2 \sum_{\substack{g \geq 1 \\ \vec{i} \geq \vec{0}}}\left( \sum_{\substack{g_1 + g_2 = g - 1 \\ 0 \leq g_2 \leq g - 1 \\ \vec{\ell} \leq \vec{i}}} (-1)^{|\vec{\ell}|}{2g \choose 2g_2 + 2} d_{(\vec{i} - \vec{\ell}), 2g_1 + 2}d_{(\vec{\ell}), 2g_2 + 2} \right)s^{\vec{i}}\frac{t^{2g}}{(2g)!}
\end{align*}

\noindent The right hand side of Equation \ref{thepdes2} is

\begin{align*}
2\partial_t^3F(\vec{s}, t)\partial_tF(-\vec{s}, t) - \partial_t^2G(\vec{s}, t) & = 2\left( \sum_{\substack{g \geq 1 \\ \vec{i} \geq \vec{0}}} D_{(\vec{i}), 2g + 2} s^{\vec{i}}\frac{t^{2g - 1}}{(2g - 1)!} \right) \left( \sum_{\substack{g \geq 0 \\ \vec{i} \geq \vec{0}}} (-1)^{|\vec{i}|}D_{(\vec{i}), 2g + 2}s^{\vec{i}}\frac{t^{2g + 1}}{(2g + 1)!} \right) \\
& - \sum_{\substack{g \geq 0 \\ \vec{i} \geq \vec{0}}}d_{(\vec{i}), 2g + 2}s^{\vec{i}}\frac{t^{2g}}{(2g)!} \\
& = \sum_{\substack{g \geq 1 \\ \vec{i} \geq \vec{0}}}\left( \left(\sum_{\substack{g_1 + g_2 = g \\ 1 \leq g_2 \leq g - 1 \\ \vec{\ell} \leq \vec{i}}} (-1)^{|\vec{\ell}|}{2g \choose 2g_2 + 1} D_{(\vec{i} - \vec{\ell}), 2g_1 + 2}D_{(\vec{\ell}), 2g_2 + 2} \right)  - d_{(\vec{i}), 2g + 2}\right)s^{\vec{i}}\frac{t^{2g}}{(2g)!}
\end{align*}

\noindent Therefore, Equation \ref{thepdes2} follows from Equation \ref{NPHuntwisted} in Theorem \ref{NPHrecursion} with $k = 0$.
\end{proof}


\section{Example}\label{example}

\noindent To demonstrate the practical use of the theorems in this paper, we go through an example of computing $D_{(\vec{i}), 2g + 2}$ and $d_{(\vec{i}), 2g + 2}$ as polynomials in $g$. Since we know that $2^{|\vec{i}| + 1}D_{(\vec{i}), 2g + 2}$ and $2^{|\vec{i}| + 1}d_{(\vec{i}), 2g + 2}$ are integer valued polynomials in $g$, these polynomials can be written uniquely as linear combinations of

\begin{equation*}
{g \choose 0}, {g \choose 1}, \ldots, {g \choose |\vec{i}|^2 + 1}
\end{equation*}

\noindent where the coefficients are integers. This is summarized in the following Corollary:


\begin{Corollary}\label{binomialbasis}
Let $\vec{i} = (i_1, \ldots, i_n) \in \mathbb{Z}^n_{\geq 0}$. Define the integers $m_{\vec{i}}$ and $\widetilde{m}_{\vec{i}}$ by

\begin{align*}
& m_{\vec{i}} := \text{min}\left\{g : \int_{\mbar_{0, (2g + 2)t}(\bztwo)}(\psi_j)^{2g - 1 - |\vec{i}|} \lambda_{i_1}\ldots\lambda_{i_n} \not = 0\right\} \\
& \widetilde{m}_{\vec{i}} := \text{min}\left\{ g : \int_{\mbar_{0, (2g + 2)t, 1u}(\bztwo)}(\psi_{2g + 3})^{2g - |\vec{i}|} \lambda_{i_1}\ldots\lambda_{i_n} \not = 0\right\}
\end{align*}

\noindent Then there exists integers $c_{\vec{i}, k}$ and $\widetilde{c}_{\vec{i}, k}$, where $0 \leq k \leq |\vec{i}|^2 + 1$, such that, for all $g$, 

\begin{align*}
& 2^{|\vec{i}| + 1}\int_{\mbar_{0, (2g + 2)t}(\bztwo)}(\psi_j)^{2g - 1 - |\vec{i}|} \lambda_{i_1}\ldots\lambda_{i_n} = \sum_{k = m_{\vec{i}}}^{|\vec{i}|^2 + 1}c_{\vec{i}, k}{g \choose k} \\
& 2^{|\vec{i}| + 1}\int_{\mbar_{0, (2g + 2)t, 1u}(\bztwo)}(\psi_{2g + 3})^{2g - |\vec{i}|} \lambda_{i_1}\ldots\lambda_{i_n} = \sum_{k = \widetilde{m}_{\vec{i}}}^{|\vec{i}|^2 + 1}\widetilde{c}_{\vec{i}, k}{g \choose k}
\end{align*}

\noindent With the aid of Theorem \ref{NPHrecursion} and Equation \ref{PHrecursion}, the integers $c_{\vec{i}, k}$ and $\widetilde{c}_{\vec{i}, k}$ are computed recursively by

\begin{align*}
& c_{\vec{i}, 0} = 2^{|\vec{i}| + 1}\int_{\mbar_{0, 2t}(\bztwo)}(\psi_j)^{2g - 1 - |\vec{i}|} \lambda_{i_1}\ldots\lambda_{i_n} \\
& \widetilde{c}_{\vec{i}, 0} = 2^{|\vec{i}| + 1}\int_{\mbar_{0, 2t, 1u}(\bztwo)}(\psi_{2g + 3})^{2g - |\vec{i}|} \lambda_{i_1}\ldots\lambda_{i_n} \\
(k > 0) \hspace{1cm} & c_{\vec{i}, k} = 2^{|\vec{i}| + 1}\int_{\mbar_{0, (2k + 2)t}(\bztwo)}(\psi_j)^{2g - 1 - |\vec{i}|} \lambda_{i_1}\ldots\lambda_{i_n} - \sum_{j = 0}^{k - 1}c_{\vec{i}, j}{k \choose j} \\
(k > 0) \hspace{1cm} & \widetilde{c}_{\vec{i}, k} = 2^{|\vec{i}| + 1}\int_{\mbar_{0, (2k + 2)t, 1u}(\bztwo)}(\psi_{2g + 3})^{2g - |\vec{i}|} \lambda_{i_1}\ldots\lambda_{i_n} - \sum_{j = 0}^{k - 1}\widetilde{c}_{\vec{i}, j}{k \choose j}
\end{align*}
\end{Corollary}

It is difficult to implement our recursions by hand, so instead, the author has written procedures in Maple to aid in this computation. The code is available at \cite{Website}.

\begin{Example}
$\vec{i} = (1, 2)$ \\

\noindent Since $|\vec{i}| = 3$ and $m_{(1, 2)} = \widetilde{m}_{(1, 2)} = 2$, we have

\begin{align*}
& 2^4D_{(1, 2), 2g + 2} = \sum_{k = 2}^{10}c_{(1, 2), k}{g \choose k} \\
& 2^4d_{(1, 2), 2g + 2} = \sum_{k = 2}^{10}\widetilde{c}_{(1, 2), k}{g \choose k}
\end{align*}

\noindent Using the recursions for $c_{\vec{i}, k}$ and $\widetilde{c}_{\vec{i}, k}$ in Corollary \ref{binomialbasis}, we have

\begin{align*}
& c_{(1, 2), 2} = 2^4D_{(1, 2), 6} = 2 \\
& c_{(1, 2), 3} = 2^4D_{(1, 2), 8} - 2{3 \choose 2} = 61 \\
& c_{(1, 2), 4} = 2^4D_{(1, 2), 10} - 62{4 \choose 3} - 2{4 \choose 2} = 364 \\
& c_{(1, 2), 5} = 2^4D_{(1, 2), 12} - 364{5 \choose 4} - 61{5 \choose 3} - 2{5 \choose 2} = 660 \\
& c_{(1, 2), 6} = 2^4D_{(1, 2), 14} - 660{6 \choose 5} - 364{6 \choose 4} - 61{6 \choose 3} - 2{6 \choose 2} = 360 \\
& c_{(1, 2), 7} = c_{(1, 2), 8} = c_{(1, 2), 9} = c_{(1, 2), 10} = 0
\end{align*}

\noindent A similar calculation results in 

\begin{equation*}
\{\widetilde{c}_{(1, 2), k} \}_{k \geq 2} = \{8, 168, 640, 840, 360, 0, 0, \ldots \}
\end{equation*}
\end{Example}

\noindent Therefore, we have

\begin{align*}
& 2^4D_{(1, 2), 2g + 2} = 2{g \choose 2} + 61{g \choose 3} + 364{g \choose 4} + 660{g \choose 5} + 360{g \choose 6} \\
& 2^4d_{(1, 2), 2g + 2} = 8{g \choose 2} + 168{g \choose 3} + 640{g \choose 4} + 840{g \choose 5} + 360{g \choose 6}
\end{align*}


\section{Conjectures, Open Problems}\label{conjectures}

The combinatorial structure that governs the $\mathbb{Z}_2$ Hurwitz-Hodge integrals $D_{(\vec{i}, 2g + 2}$ and $d_{(\vec{i}), 2g + 2}$ is far from being completely understood. Here, we mention some out standing problems concerning these intersection numbers, and outline some of the avenues of investigation for future work. 

By Theorem \ref{Polynomiality}, we know that $D_{(\vec{i}), 2g + 2}$ and $d_{(\vec{i}), 2g + 2}$ are polynomials in $g$ of degree at most $|\vec{i}|^2 + 1$. However, the data in the Appendix suggests the following much sharper bound:

\begin{Conjecture}\label{polynomialconjecture}
The integrals $D_{(\vec{i}), 2g + 2}$ and $d_{(\vec{i}), 2g + 2}$ are polynomials in $g$, and their degrees are precisely $2|\vec{i}|$
\end{Conjecture}

Let us take a moment to explain what happens when one tries to use the proof techniques in the proof of Theorem \ref{Polynomiality} to tackle Conjecture \ref{polynomialconjecture}. When we try to sharpen the bound on the degrees of the polynomials $D_{(\vec{i}), 2g + 2}$ and $d_{(\vec{i}), 2g + 2}$ to $2|\vec{i}|$, the only difference in the proof would be the induction step, and Equations \ref{KeyStep1} and \ref{KeyStep2} become

\begin{align*}
& \widetilde{F}_{\vec{i}}(s) = A_{\vec{i}, 2}\sum_{k = 0}^{2(|\vec{i}| - 1)}\frac{b_k}{(s - 1)^{|\vec{i}| + k + 2}} + A_{\vec{i}, 3}\sum_{k = 0}^{2(|\vec{i}| - 1)}\frac{a_k}{(s - 1)^{|\vec{i}| + k +2}} \\
& \widetilde{G}_{\vec{i}}(s) = A_{\vec{i}, 4}\sum_{k = 0}^{2(|\vec{i}| - 1)}\frac{a_k}{(s - 1)^{|\vec{i}| + k + 2}} + A_{\vec{i}, 1}\sum_{k = 0}^{2(|\vec{i}| - 1)}\frac{b_k}{(s - 1)^{|\vec{i}| + k + 2}}
\end{align*}

\noindent and therefore,

\begin{align*}
& \widetilde{F}_{\vec{i}}(s) = \frac{\displaystyle \sum_{k = 0}^{2(|\vec{i}| - 1)}(A_{\vec{i}, 2}b_k + A_{\vec{i}, 3}a_k)(s - 1)^{2(|\vec{i}| - 1) - k}}{(s - 1)^{3|\vec{i}|}} = \frac{\displaystyle \sum_{k = 0}^{2(|\vec{i}| - 1)}(A_{\vec{i}, 2}b_k + A_{\vec{i}, 3}a_k)(s - 1)^{|\vec{i}| - 1 - k}}{(s - 1)^{2|\vec{i}| + 1}} \\
& \widetilde{G}_{\vec{i}}(s) = \frac{\displaystyle \sum_{k = 0}^{2(|\vec{i}| - 1)}(A_{\vec{i}, 4}a_k + A_{\vec{i}, 1}b_k)(s - 1)^{2(|\vec{i}| - 1) - k}}{(s - 1)^{3|\vec{i}|}} = \frac{\displaystyle \sum_{k = 0}^{2(|\vec{i}| - 1)}(A_{\vec{i}, 4}a_k + A_{\vec{i}, 1}b_k)(s - 1)^{|\vec{i}| - 1 - k}}{(s - 1)^{2|\vec{i}| + 1}}
\end{align*}

\noindent At this point, the proof would go through as before only if the numerators are guaranteed to be \emph{polynomials}, but we can't guarantee this since the expressions in the numerators have poles at $s = 1$ for $k > |\vec{i}| - 1$. 

In fact, the techniques used in the proof of Theorem \ref{Polynomiality} do not accommodate any linear bound in $|\vec{i}|$. Indeed, if we wanted to prove that the polynomial degrees were bounded by $n|\vec{i}|$, then Equations \ref{KeyStep1} and \ref{KeyStep2} become

\begin{align*}
& \widetilde{F}_{\vec{i}}(s) = A_{\vec{i}, 2}\sum_{k = 0}^{n(|\vec{i}| - 1)}\frac{b_k}{(s - 1)^{|\vec{i}| + k + 2}} + A_{\vec{i}, 3}\sum_{k = 0}^{n(|\vec{i}| - 1)}\frac{a_k}{(s - 1)^{|\vec{i}| + k +2}} \\
& \widetilde{G}_{\vec{i}}(s) = A_{\vec{i}, 4}\sum_{k = 0}^{n(|\vec{i}| - 1)}\frac{a_k}{(s - 1)^{|\vec{i}| + k + 2}} + A_{\vec{i}, 1}\sum_{k = 0}^{n(|\vec{i}| - 1)}\frac{b_k}{(s - 1)^{|\vec{i}| + k + 2}}
\end{align*}

and therefore,

\begin{align*}
& \widetilde{F}_{\vec{i}}(s) = \frac{\displaystyle \sum_{k = 0}^{n(|\vec{i}| - 1)}(A_{\vec{i}, 2}b_k + A_{\vec{i}, 3}a_k)(s - 1)^{n(|\vec{i}| - 1) - k}}{(s - 1)^{(n + 1)|\vec{i}| + (2 - n)}} = \frac{\displaystyle \sum_{k = 0}^{n(|\vec{i}| - 1)}(A_{\vec{i}, 2}b_k + A_{\vec{i}, 3}a_k)(s - 1)^{(n - 1)|\vec{i}| - 1 - k}}{(s - 1)^{n|\vec{i}| + 1}} \\
& \widetilde{G}_{\vec{i}}(s) = \frac{\displaystyle \sum_{k = 0}^{n(|\vec{i}| - 1)}(A_{\vec{i}, 4}a_k + A_{\vec{i}, 1}b_k)(s - 1)^{n(|\vec{i}| - 1) - k}}{(s - 1)^{(n + 1)|\vec{i}| + (2 - n)}} = \frac{\displaystyle \sum_{k = 0}^{n(|\vec{i}| - 1)}(A_{\vec{i}, 4}a_k + A_{\vec{i}, 1}b_k)(s - 1)^{(n - 1)|\vec{i}| - 1 - k}}{(s - 1)^{n|\vec{i}| + 1}}
\end{align*}

\noindent in which case we still cannot guarantee that the expression in the numerators are polynomials due to poles at $s = 1$ when $k > (n - 1)|\vec{i}| - 1$. Alas, proving Conjecture \ref{polynomialconjecture} will require a different approach. 

By Theorem \ref{Integrality} we know that $|\vec{i}| + 1$ is a \emph{sufficient} power of $2$ that will normalize the Hodge integrals so that they are integral. However, the exponent $|\vec{i}| + 1$ is far from being \emph{necessary}. For example, using the recursions, one can compute 

\begin{equation*}
\int_{\overline{\mathcal{H}}_{4, 10}}\lambda_1\lambda_2\lambda_4 =: D_{(1, 2, 4), 10} = \frac{27}{8}
\end{equation*}

\noindent Certainly, $2^{(1 + 2 + 4) + 1}D_{(1,2,4), 10} \in \mathbb{Z}$, but $3$ is the \emph{smallest} exponent for which the integrality holds. Thus, we have the following open problem:

\begin{Open Problem}
Compute the integers $m_1(\vec{i}, g)$ and $m_2(\vec{i}, g)$, where

\begin{align*}
& m_1(\vec{i}, g) := \text{min}\left \{\alpha : 2^{\alpha} D_{(\vec{i}), 2g + 2} \in \mathbb{Z} \right\} \\
& m_2(\vec{i}, g) := \text{min}\left \{\alpha : 2^{\alpha} d_{(\vec{i}), 2g + 2} \in \mathbb{Z} \right\}
\end{align*}
\end{Open Problem}

\noindent By Corollary \ref{binomialbasis}, we know that there exists integers $c_{\vec{i}, k}$ and $\widetilde{c}_{\vec{i}, k}$ such that

\begin{align}
& 2^{|\vec{i}| + 1}D_{(\vec{i}), 2g + 2} = \sum_{k = 0}^{|\vec{i}|^2 + 1}c_{\vec{i}, k}{g \choose k} \label{binomialbasis1} \\
& 2^{|\vec{i}| + 1}d_{(\vec{i}), 2g + 2} = \sum_{k = 0}^{|\vec{i}|^2 + 1}\widetilde{c}_{\vec{i}, k}{g \choose k} \label{binomialbasis2}
\end{align}

\noindent Many examples of the sequences $\{c_{\vec{i}, k}\}_{k}$ and $\{\widetilde{c}_{\vec{i}, k}\}_k$ are shown in the Appendix, and from this data, we make the following conjecture: 

\begin{Conjecture}\label{unimodalconjecture}
Let $c_{\vec{i}, k}$ and $\widetilde{c}_{\vec{i}, k}$ be the integers defined in Equations \ref{binomialbasis1} and \ref{binomialbasis2}. Then

\begin{enumerate}
\item{The integers $c_{\vec{i}, k}$ and $\widetilde{c}_{\vec{i}, k}$ are $\geq 0$} 
\item{For fixed $\vec{i}$, the sequences $\{c_{\vec{i}, k}\}_{k \geq 0}$ and $\{\widetilde{c}_{\vec{i}, k}\}_{k \geq 0}$ are log-concave}
\end{enumerate}
\end{Conjecture}

\noindent In (\cite{AbelianHurwitzHodge2011},Theorem 2), Johnson et al discover the following vanishing result:

\begin{equation*}
\sum_{i = 0}^g(-2)^iD_{i, 2g + 2} = 0
\end{equation*}

\noindent This follows from the results in \cite{Afandi2019}. However, using the data in the Appendix, it is likely that the above vanishing result can be generalized:

\begin{Conjecture}
Let $g > 0$ and let $\ell \in \mathbb{Z}^n_{\geq 0}$ such that $|\vec{\ell}| \leq g - 1$. Then

\begin{enumerate}
\item{$\displaystyle \sum_{i = 0}^g(-2)^iD_{(\vec{\ell}, i), 2g + 2}  =  0$}
\item{The sequences $\left\{2^{|\vec{\ell}| + i + 1}D_{(\vec{\ell}, i), 2g + 2}\right\}_{i \geq 0}$ and $\left\{2^{|\vec{\ell}| + i + 1}d_{(\vec{\ell}, i), 2g + 2}\right\}_{i \geq 0}$ are log-concave. } 
\end{enumerate}
\end{Conjecture}

\noindent Finally, recall the result of Faber and Pandharipande \cite{FP2000}:

\begin{equation*}
D_{(g-1, g), 2g + 2} = \frac{2^{2g} - 1}{2g}|B_{2g}|
\end{equation*}

\noindent where $B_{2g}$ is the $2g^{th}$ Bernoulli number. By Theorem \ref{Integrality}, we know that $2^{2g}D_{(g - 1, g), 2g + 2} \in \mathbb{Z}$. It turns out that the integer $2^{2g}D_{(g - 1, g), 2g + 2}$ has an enumerative meaning:

\begin{Corollary}\label{zigzag}
Let $Z_g$ be the number of alternating permutations on $[1, 2, \ldots, 2g - 1]$. Then

\begin{equation*}
2^{2g}D_{(g - 1, g), 2g + 2} = Z_g
\end{equation*}

\end{Corollary}

\noindent This observation prompts the following line of questioning:

\begin{Question}
Does there exist an enumerative interpretation of the integers $2^{|\vec{i}| + 1}D_{(\vec{i}), 2g + 2}$ and $2^{|\vec{i}| + 1}d_{(\vec{i}), 2g + 2}$ as counting subsets of permutations of $[1, 2, \ldots, 2g - 1]$, as in Corollary \ref{zigzag}? As a starting point, in light of the results obtained in \cite{Afandi2019}, is there a way to interpret the integers $e_i(1, 3, \ldots, 2g - 1)$ and $e_i(2, 4, \ldots, 2g)$ in this way?
\end{Question}

\noindent An affirmative answer to this question would provide an elegant governing principle for $\mathbb{Z}_2$ Hurwitz-Hodge integrals. Some new ideas, either coming from combinatorics or geometry, will be required to approach these questions.


\appendix

\section{$\mathbb{Z}_2$ Hurwitz-Hodge Integrals}

Using Theorem \ref{NPHrecursion} and Equation \ref{PHrecursion}, we compute many examples of $\mathbb{Z}_2$ Hurwitz-Hodge integrals. We implement the computations in Maple. The code can be found at the author's website \cite{Website}. \\

\begin{center}
\noindent \underline{{\bf{g = 1}}}
\end{center}

\begin{center}
\begin{tabular}{|c|c|}
\hline
Hodge Monomial $\lambda_{\vec{i}}$ & $D_{(\vec{i}), 2(g) + 2}$ \\
\hline
$\lambda_1$ & $\frac{1}{4}$ \\
\hline
\end{tabular}
\quad
\begin{tabular}{|c|c|}
\hline
Hodge Monomial $\lambda_{\vec{i}}$ & $d_{(\vec{i}), 2(g) + 2}$ \\
\hline
$\lambda_1$ & $\frac{1}{2}$ \\
\hline
\end{tabular}
\end{center}

\begin{center}
\underline{{\bf{g = 2}}}
\end{center}

\begin{center}
\begin{tabular}{|c|c|}
\hline
Hodge Monomial $\lambda_{\vec{i}}$ & $D_{(\vec{i}), 2(g) + 2}$ \\
\hline
$\lambda_1$ & $1$ \\
\hline
$\lambda_2$ &  $\frac{3}{8}$ \\
\hline
$\lambda_1\lambda_2$ & $\frac{1}{8}$ \\
\hline
\end{tabular}
\quad
\begin{tabular}{|c|c|}
\hline
Hodge Monomial $\lambda_{\vec{i}}$ & $d_{(\vec{i}), 2g + 2}$ \\
\hline
$\lambda_1$ & $\frac{3}{2}$ \\
\hline
$\lambda_2$ & $1$ \\
\hline
$\lambda_1\lambda_2$ & $\frac{1}{2}$ \\
\hline
\end{tabular}
\end{center}

\begin{center}
\underline{{\bf{g = 3}}}
\end{center}

\begin{center}
\begin{tabular}{|c|c|}
\hline
Hodge Monomial $\lambda_{\vec{i}}$ & $D_{(\vec{i}), 2g + 2}$ \\
\hline
$\lambda_1$ & $\frac{9}{4}$ \\
\hline
$\lambda_2$ & $\frac{23}{8}$ \\
\hline
$\lambda_3$ & $\frac{15}{16}$ \\
\hline
$\lambda_1\lambda_2$ & $\frac{67}{16}$ \\
\hline
$\lambda_1\lambda_3$ & $\frac{15}{16}$ \\
\hline
$\lambda_2\lambda_3$ & $\frac{1}{4}$ \\
\hline
\end{tabular}
\quad
\begin{tabular}{|c|c|}
\hline
Hodge Monomial $\lambda_{\vec{i}}$ & $d_{(\vec{i}), 2g + 2}$ \\
\hline
$\lambda_1$ & $3$ \\
\hline
$\lambda_2$ & $\frac{11}{2}$ \\
\hline
$\lambda_3$ & $3$ \\
\hline
$\lambda_1\lambda_2$ & $12$ \\
\hline
$\lambda_1\lambda_3$ & $4$  \\
\hline
$\lambda_2\lambda_3$ & $\frac{3}{2}$ \\
\hline
\end{tabular}
\end{center}

\pagebreak

\begin{center}
\underline{{\bf{g = 4}}}
\end{center}

\begin{center}
\begin{tabular}{|c|c|}
\hline
Hodge Monomial $\lambda_{\vec{i}}$ & $D_{(\vec{i}), 2g + 2}$ \\
\hline
$\lambda_1$ & $4$ \\
\hline
$\lambda_2$ & $\frac{43}{4}$ \\
\hline
$\lambda_3$ & $11$ \\
\hline
$\lambda_4$ & $\frac{105}{32}$ \\
\hline
$\lambda_1\lambda_2$ & $\frac{155}{4}$ \\
\hline 
$\lambda_1\lambda_3$ & $\frac{221}{8}$ \\
\hline
$\lambda_1\lambda_4$ & $\frac{105}{16}$ \\
\hline
$\lambda_2\lambda_3$ & $\frac{403}{16}$ \\
\hline
$\lambda_2\lambda_4$ & $\frac{147}{32}$ \\
\hline
$\lambda_3\lambda_4$ & $\frac{17}{16}$ \\
\hline 
$\lambda_1\lambda_2\lambda_4$ & $\frac{27}{8}$ \\
\hline
\end{tabular}
\quad
\begin{tabular}{|c|c|}
\hline
Hodge Monomial $\lambda_{\vec{i}}$ & $d_{(\vec{i}), 2g + 2}$ \\
\hline
$\lambda_1$ & $5$ \\
\hline
$\lambda_2$ & $\frac{35}{2}$ \\
\hline
$\lambda_3$ & $25$ \\
\hline
$\lambda_4$ & $12$ \\
\hline
$\lambda_1\lambda_2$ & $85$ \\
\hline 
$\lambda_1\lambda_3$ & $85$ \\
\hline
$\lambda_1\lambda_4$ & $30$ \\
\hline
$\lambda_2\lambda_3$ & $\frac{211}{2}$ \\
\hline
$\lambda_2\lambda_4$ & $27$ \\
\hline
$\lambda_3\lambda_4$ & $\frac{17}{2}$ \\
\hline 
$\lambda_1\lambda_2\lambda_4$ & $27$ \\
\hline
\end{tabular}
\end{center}

\begin{center}
\underline{{\bf{g = 5}}}
\end{center}

\begin{center}
\begin{tabular}{|c|c|}
\hline
Hodge Monomial $\lambda_{\vec{i}}$ & $D_{(\vec{i}), 2g + 2}$ \\
\hline
$\lambda_1$ & $\frac{25}{4}$ \\
\hline
$\lambda_2$ & $\frac{115}{4}$ \\
\hline
$\lambda_3$ & $\frac{475}{8}$ \\
\hline
$\lambda_4$ & $\frac{1689}{32}$ \\
\hline
$\lambda_5$ & $\frac{945}{64}$ \\
\hline
$\lambda_1\lambda_2$ & $\frac{1555}{8}$ \\
\hline 
$\lambda_1\lambda_3$ & $\frac{1195}{4}$ \\
\hline
$\lambda_1\lambda_4$ & $\frac{13185}{64}$ \\
\hline
$\lambda_1\lambda_5$ & $\frac{1575}{32}$ \\
\hline
$\lambda_2\lambda_3$ & $\frac{18599}{32}$ \\
\hline
$\lambda_2\lambda_4$ & $\frac{10179}{32}$ \\
\hline
$\lambda_2\lambda_5$ & $\frac{4095}{64}$ \\
\hline
$\lambda_3\lambda_4$ & $\frac{14801}{64}$ \\
\hline 
$\lambda_3\lambda_5$ & $\frac{1185}{32}$ \\
\hline
$\lambda_4\lambda_5$ & $\frac{31}{4}$ \\
\hline
$\lambda_1\lambda_2\lambda_3$ & $\frac{56119}{32}$ \\
\hline
$\lambda_1\lambda_2\lambda_4$ & $\frac{47367}{64}$ \\
\hline
$\lambda_1\lambda_2\lambda_5$ & $\frac{1845}{16}$ \\
\hline
$\lambda_1\lambda_3\lambda_4$ & $\frac{11835}{32}$ \\
\hline
$\lambda_1\lambda_3\lambda_5$ & $\frac{139}{4}$ \\
\hline
$\lambda_2\lambda_3\lambda_4$ & $\frac{1381}{8}$ \\
\hline
\end{tabular}
\quad
\begin{tabular}{|c|c|}
\hline
Hodge Monomial $\lambda_{\vec{i}}$ & $d_{(\vec{i}), 2g + 2}$ \\
\hline
$\lambda_1$ & $\frac{15}{2}$ \\
\hline
$\lambda_2$ & $\frac{85}{2}$ \\
\hline
$\lambda_3$ & $\frac{225}{2}$ \\
\hline
$\lambda_4$ & $137$ \\
\hline
$\lambda_5$ & $60$ \\
\hline
$\lambda_1\lambda_2$ & $\frac{725}{2}$ \\
\hline 
$\lambda_1\lambda_3$ & $725$ \\
\hline
$\lambda_1\lambda_4$ & $680$ \\
\hline
$\lambda_1\lambda_5$ & $240$ \\
\hline
$\lambda_2\lambda_3$ & $\frac{3637}{2}$ \\
\hline
$\lambda_2\lambda_4$ & $\frac{2687}{2}$ \\
\hline
$\lambda_2\lambda_5$ & $381$ \\
\hline
$\lambda_3\lambda_4$ & $1279$ \\
\hline 
$\lambda_3\lambda_5$ & $278$ \\
\hline
$\lambda_4\lambda_5$ & $\frac{155}{2}$ \\
\hline
$\lambda_1\lambda_2\lambda_3$ & $\frac{14295}{2}$ \\
\hline
$\lambda_1\lambda_2\lambda_4$ & $4087$ \\
\hline
$\lambda_1\lambda_2\lambda_5$ & $864$ \\
\hline
$\lambda_1\lambda_3\lambda_4$ & $2762$ \\
\hline
$\lambda_1\lambda_3\lambda_5$ & $\frac{695}{2}$ \\
\hline
$\lambda_2\lambda_3\lambda_4$ & $\frac{6905}{4}$ \\
\hline
\end{tabular}
\end{center}

\section{Integer Valued Polynomials}

In this section, we compile tables of various examples of the integer valued polynomials $2^{|\vec{i}| + 1}D_{(\vec{i}), 2g + 2}$ and $2^{|\vec{i}| + 1}d_{(\vec{i}), 2g + 2}$. If the reader decides to read the data, we recommend reading the data in light of of Conjecture \ref{polynomialconjecture} and Conjecture \ref{unimodalconjecture}.

\begin{center}
\begin{tabular}{|c|c|}
\hline
Hodge Monomial $\lambda_{\vec{i}}$ & $2^{|\vec{i}| + 1}\cdot D_{(\vec{i}), 2g + 2}$ \\
\hline
 $\lambda_1$ &  $\displaystyle {g \choose 2} + 2{g \choose 2}$ \\
 \hline 
 $\lambda_2$ & $\displaystyle 3{g \choose 2} + 14{g \choose 3} + 12{g \choose 4}$\\
 \hline
 $\lambda_3$ & $\displaystyle 15{g \choose 3} +  116{g \choose 4} + 220{g \choose 5} + 120{g \choose 6}$  \\
 \hline
 $\lambda_4$ & $\displaystyle 105{g \choose 4} + 1164{g \choose 5} + 3580{g \choose 6} + 4200{g \choose 7} + 1680{g \choose 8}$ \\
 \hline
 $\lambda_5$ & $\displaystyle 945{g \choose 5} + 13854{g \choose 6} + 60508{g \choose 7} + 113120{g \choose 8} + 95760{g \choose 9} + 30240{g \choose 10}$ \\
 \hline
\end{tabular}
\end{center}

\begin{center}
\begin{tabular}{|c|c|}
\hline
Hodge Monomial $\lambda_{\vec{i}}$ & $2^{|\vec{i}| + 1}\cdot d_{(\vec{i}), 2g + 2}$ \\
\hline 
$\lambda_1$ & $\displaystyle 2{g \choose 1} + 2{g \choose 2}$ \\
 \hline 
 $\lambda_2$ & $\displaystyle 8{g \choose 2} + 20{g \choose 3} + 12{g \choose 4}$  \\
 \hline
 $\lambda_3$ & $\displaystyle 48{g \choose 3} + 208{g \choose 4} + 280{g \choose 5} + 120{g \choose 6}$ \\
 \hline
 $\lambda_4$ & $\displaystyle 384{g \choose 4} + 2464{g \choose 5} + 5440{g \choose 6} + 5040{g \choose 7} + 1680{g \choose 8}$ \\
 \hline
 $\lambda_5$ & $\displaystyle 3840{g \choose 5} + 33408{g \choose 6} + 105728{g \choose 7} + 156800{g  \choose 8} + 110880{g \choose 9} + 30240{g \choose 10}$ \\
 \hline
\end{tabular}
\end{center}

\begin{center}
\begin{tabular}{|c|c|}
\hline
Hodge Monomial $\lambda_{\vec{i}}$ & $2^{|\vec{i}| + 1}\cdot D_{(\vec{i}), 2g + 2}$ \\
\hline 
$\lambda_1\lambda_2$ & $\displaystyle 2{g \choose 2} + 61{g \choose 3} + 364{g \choose 4} + 660{g \choose 5} + 360{g \choose 6}$  \\
\hline 
$\lambda_1\lambda_3$ & $\displaystyle 30{g \choose 3} + 764{g \choose 4} + 5440{g \choose 5} + 14800{g \choose 6} + 16800{g \choose 7} + 6720 {g \choose 8} $ \\
\hline 
$\lambda_2\lambda_3$ & $\begin{array}{l} \displaystyle 16{g \choose 3} + 1548{g \choose 4} + 29298{g \choose 5} + 208724{g \choose 6}  \\ \displaystyle + 697144{g \choose 7} + 1171520{g \choose 8} + 957600{g \choose 9} + 302400{g \choose 10}\end{array}$ \\
\hline 
$\lambda_2\lambda_4$ & $\begin{array}{l} \displaystyle 588{g \choose 4} +  37776{g \choose 5} + 688661{g \choose 6} + 5395054{g \choose 7} + 21681016{g \choose 8} \\ \displaystyle + 48109152{g \choose 9} + 59446800{g \choose 10} + 38253600{g \choose 11} + 9979200{g \choose 12} \end{array}$ \\
\hline
$\lambda_3\lambda_4$ & $\begin{array}{l} \displaystyle 272{g \choose 4} + 57844{g \choose 5} + 2318756{g \choose 6} + 36063149{g \choose 7} + 281352536{g \choose 8} \\ \displaystyle + 1243038072{g \choose 9} + 3296287120{g \choose 10} + 5336685200{g \choose 11} \\ \displaystyle + 5154811200{g \choose 12} + 2724321600{g \choose 13} + 605404800{g \choose 14} \end{array}$ \\
\hline
$\lambda_3\lambda_5$ & $\begin{array}{l} \displaystyle 18960{g \choose 5} + 2372368{g \choose 6} + 82990414{g \choose 7} + 1277796904{g \choose 8} + 10577739904{g \choose 9} \\
\displaystyle + 52202352960{g \choose 10} + 162223105440{g \choose 11} + 324989181440{g \choose 12} + 418367389440{g \choose 13} \\ \displaystyle +  333860567040{g \choose 14} +  150140390400{g \choose 15} + 29059430400{g \choose 16} \end{array}$\\
\hline
$\lambda_4\lambda_5$ & $\begin{array}{l} \displaystyle 7936{g \choose 5} + 3077720{g \choose 6} + 218029720{g \choose 7} + 6007887736{g \choose 8} + 84538397486{g \choose 9} \\ \displaystyle + 696641555804{g \choose 10} + 3632616132464{g \choose 11} + 12530730860032{g \choose 12} \\ \displaystyle  + 29238584620960{g \choose 13} + 46328057455680{g \choose 14} + 49056314186880{g \choose 15} \\ \displaystyle + 33195555993600{g \choose 16} + 12967770816000{g \choose 17} + 2223046425600{g \choose 18} \end{array}$\\
\hline
\end{tabular}
\end{center}

\begin{center}
\begin{tabular}{|c|c|}
\hline
Hodge Monomial $\lambda_{\vec{i}}$ & $2^{|\vec{i}| + 1}d_{(\vec{i}), 2g + 2}$ \\
\hline 
$\lambda_1\lambda_2$ & $ \displaystyle 8{g \choose 2} + 168{g \choose 3} + 640{g \choose 4} + 840{g \choose 5} + 360{g \choose 6}$\\
\hline 
$\lambda_1\lambda_3$ & $ \displaystyle 128{g \choose 3} + 2208{g \choose 4} + 10880{g \choose 5} + 22240{g \choose 6} + 20160{g \choose 7} + 6720{g \choose 8}$ \\
\hline 
$\lambda_2\lambda_3$ & $ \begin{array}{c} \displaystyle 96{g \choose 3} + 6368{g \choose 4} + 83584{g \choose 5} + 444896{g \choose 6} \\ \displaystyle + 1169504{g \choose 7} + 1608320{g \choose 8} + 1108800{g \choose 9} + 302400{g \choose 10} \end{array}$\\
\hline 
$\lambda_2\lambda_4$ & $ \begin{array}{c} \displaystyle 3456{g \choose 4} + 154688{g \choose 5} + 2037312{g \choose 6} + 12279680{g \choose 7} \\ \displaystyle + 39703552{g \choose 8} + 73258752{g \choose 9} + 77212800{g \choose 10} + 43243200{g \choose 11} + 9979200{g \choose 12} \end{array}$ \\
\hline
$\lambda_3\lambda_4$ & $ \begin{array}{c} \displaystyle 2176{g \choose 4} + 316544{g \choose 5} + 8989696{g \choose 6} + 106257664{g \choose 7} \\ \displaystyle + 661451008{g \choose 8} + 2413320192{g \choose 9} + 5421758720{g \choose 10} + 7586163200{g \choose 11} \\ \displaystyle + 6435475200{g \choose 12} + 3027024000{g \choose 13} + 605404800{g \choose 14}\end{array}$\\
\hline
$\lambda_3\lambda_5$ & $\begin{array}{c} \displaystyle 142336{g \choose 5} + 12537344{g \choose 6} + 321875456{g \choose 7} + 3857878016{g \choose 8} \\ \displaystyle + 25934706688{g \choose 9} + 107151344640{g \choose 10} + 285191920640{g \choose 11} \\ \displaystyle + 498116917760{g \choose 12} + 567135528960{g \choose 13} + 405056171520{g \choose 14} \\ \displaystyle  + 164670105600{g \choose 15} + 29059430400{g \choose 16}\end{array}$\\
\hline
$\lambda_4\lambda_5$ & $\begin{array}{c} \displaystyle 79360{g \choose 5} + 21071360{g \choose 6} + 1068101632{g \choose 7} + 22535566336{g \choose 8} \\ \displaystyle  + 254648976384{g \choose 9} + 1742858805248{g \choose 10} + 7737100743680{g \choose 11} \\ \displaystyle + 23154969480192{g \choose 12} + 47579454208000{g \choose 13} + 67202354419200{g \choose 14} \\ \displaystyle + 64079434298880{g \choose 15} + 39385214668800{g \choose 16} \\ \displaystyle + 14079294028800{g \choose 17} + 2223046425600{g \choose 18}\end{array}$ \\
\hline
\end{tabular}
\end{center}

\begin{center}
\begin{tabular}{|c|c|}
\hline
Hodge Monomial $\lambda_{\vec{i}}$ & $2^{|\vec{i}| + 1}D_{(\vec{i}), 2g + 2}$ \\
\hline 
$\lambda_1\lambda_2\lambda_3$ & $ \begin{array}{c} \displaystyle 3976{g \choose 4} + 204596{g \choose 5} + 3262812{g \choose 6} + 23632720{g \choose 7} \\ \displaystyle + 90667808{g \choose 8} + 195992832{g \choose 9} + 238996800{g \choose 10} \\ \displaystyle + 153014400{g \choose 11} + 39916800{g \choose 12}\end{array}$\\
\hline 
$\lambda_2\lambda_3\lambda_4$ & $\begin{array}{c} \displaystyle 176768{g \choose 5} + 52956760{g \choose 6} + 3233676672{g \choose 7} + 79792843192{g \choose 8} \\ \displaystyle + 1033784133532{g \choose 9} + 8008401061128{g \choose 10} + 39884957383392{g \choose 11} \\ \displaystyle + 133034003257408{g \choose 12} + 303050501490880{g \choose 13} + 472326815600640{g \choose 14} \\ \displaystyle + 494856672710400{g \choose 15} + 332827342848000{g \choose 16} \\ \displaystyle + 129677708160000{g \choose 17} + 22230464256000{g \choose 18}\end{array}$ \\
\hline 
\end{tabular}
\end{center}

\begin{center}
\begin{tabular}{|c|c|}
\hline
Hodge Monomial $\lambda_{\vec{i}}$ & $2^{|\vec{i}| + 1}d_{(\vec{i}), 2g + 2}$ \\
\hline 
$\lambda_1\lambda_2\lambda_3$ & $ \begin{array}{c} \displaystyle 23296{g \choose 4} + 798400{g \choose 5} + 9337664{g \choose 6} + 52746112{g \choose 7} \\ \displaystyle + 164374784{g \choose 8} + 297196032{g \choose 9} + 310060800{g \choose 10} \\ \displaystyle + 172972800{g \choose 11} + 39916800{g \choose 12}\end{array}$\\
\hline 
$\lambda_2\lambda_3\lambda_4$ & $\begin{array}{c} \displaystyle 1767680{g \choose 5} + 352174336{g \choose 6} + 15267284992{g \choose 7} + 289757178880{g \choose 8} \\ \displaystyle + 3035405595648{g \choose 9} + 19658941571072{g \choose 10} + 83824149843456{g \choose 11} \\ \displaystyle + 243660123176448{g \choose 12} + 490474167040000{g \choose 13} + 683102129149440{g \choose 14} \\ \displaystyle + 645523765286400{g \choose 15} + 394723929600000{g \choose 16} \\ \displaystyle + 140792940288000{g \choose 17} + 22230464256000{g \choose 18}\end{array}$ \\
\hline 
\end{tabular}
\end{center}

\bibliographystyle{alpha}
\bibliography{bibliography}

\end{document}